\documentclass[11pt, reqno]{amsart}
\usepackage[english]{babel}
\usepackage[utf8]{inputenc}
\usepackage{amsmath, 
amsfonts,
amsthm, 
fullpage, 
fancyhdr, 
upgreek,
amssymb,
hyperref,
enumerate,
amscd,
comment,
siunitx,
enumitem,
cite,
graphicx,
xcolor,
cleveref}
\let\oldtocsection=\tocsection

\let\oldtocsubsection=\tocsubsection

\let\oldtocsubsubsection=\tocsubsubsection

\renewcommand{\tocsection}[2]{\hspace{0em}\vspace{-0.7em}\oldtocsection{#1}{#2}}
\renewcommand{\tocsubsection}[2]{\hspace{1em}\vspace{-0.7em}\oldtocsubsection{#1}{#2}}
\renewcommand{\tocsubsubsection}[2]{\hspace{2em}\vspace{-0.7em}\oldtocsubsubsection{#1}{#2}}

\usepackage[parfill]{parskip}

\newtheorem{theorem}{Theorem}[section]
\newtheorem{proposition}[theorem]{Proposition}
\newtheorem{corollary}[theorem]{Corollary}
\newtheorem{lemma}[theorem]{Lemma}
\newtheorem{conjecture}{Conjecture}[section]
\newtheorem{example}[theorem]{Example}
\newtheorem{definition}[theorem]{Definition}
\newtheorem{remark}[theorem]{Remark}

\newcommand{\prob}{\mathbb{P}}
\newcommand{\E}{\mathbb{E}}
\newcommand{\Z}{\mathbb{Z}}
\newcommand{\R}{\mathbb{R}}
\newcommand{\C}{\mathbb{C}}

\newcommand{\ef}{\mathcal{F}}

\DeclareMathOperator{\Var}{Var}

\title{Generalized continuous and discrete stick fragmentation and Benford's law}
\author{Xinyu Fang}
\email{\textcolor{blue}{\href{mailto:fxinyu@umich.edu}{fxinyu@umich.edu}}}
\address{Department of Mathematics, National University of Singapore, Singapore}

\author{Steven J. Miller}
\email{\textcolor{blue}{\href{mailto:sjm1@williams.edu}{sjm1@williams.edu}},  \textcolor{blue}{\href{mailto:Steven.Miller.MC.96@aya.yale.edu}{Steven.Miller.MC.96@aya.yale.edu}}}
\address{Department of Mathematics and Statistics, Williams College, Williamstown, MA 01267}

\author{Maxwell Sun}
\email{\textcolor{blue}{\href{mailto:mrsun@mit.edu}{mrsun@mit.edu}}}
\address{Department of Mathematics, Massachusetts Institute of Technology, Cambridge, MA 02139}

\author{Amanda Verga}
\email{\textcolor{blue}{\href{mailto:amanda.verga@trincoll.edu}{amanda.verga@trincoll.edu}}}
\address{Department of Mathematics, Trinity College, Hartford, CT 06106}
  
\numberwithin{equation}{subsection}

\begin{document}

\maketitle

\begin{abstract}
Inspired by the basic stick fragmentation model proposed by Becker et al. in \cite{B+}, we consider three new versions of such fragmentation models, namely, continuous with random number of parts, continuous with probabilistic stopping, and discrete with congruence stopping conditions. In all of these situations, we state and prove precise conditions for the ending stick lengths to obey Benford's law when taking the appropriate limits. We introduce the \emph{aggregated limit}, necessary to guarantee convergence to Benford's law in the latter two models. We also show that resulting stick lengths are non-Benford when our conditions are not met. Moreover, we give a sufficient condition for a distribution to satisfy the Mellin transform condition introduced in 
% \cite{Jang2008ChainsOD}, 
\cite{JKKKM}, 
yielding a large family of examples.
\end{abstract}

\tableofcontents

%%%%%%%%%%%%%%%%%%%%%%%%%%%%%%%%%%%%%%%%%%%%%%%%%%%%%%%%%%%%%%%%%%%%%%%%%%%%%%%%%%%%%%%%%%%%%%%%%%%%%%%%%%%%%%%%%%%%%%%%%%%%%%%%%%%%%%%%%%%%%%%%%%%%%%%%%%%%%%%%%%%%%%%%%%%%%%%%%%%%%%%%%%%%%%%%%%%%%%%%%%%%%%%%%%%%%%%%%%%%%%%%%%%%%%%%%%%%%%%%%%%%%%%%%%%%%%%%%%%%%%%%%%%%%%%%%%%%%%%%%%%%%%%%%%%%%%%%%%%%%%%%%%%%%%%%%%%%%%%%%%%%%%%%%%%%%%%%%%%%%%%%%%%%%%%%%%%%%%%%%%%%%%%%%%%%%%%%%%%%%%%%%%%%%%%%%%%%%%%%%%%%%%%%%%%%%%%%%%%%%%%%%%%%

\section{Introduction}

%%%%%%%%%%%%%%%%%%%%%%%%%%%%%%%%%%%%%%%%%%%%%%%%%%%%%%%%%%%
%%%%%%%%%%%%%%%%%%%%%%%%%%%%%%%%%%%%%%%%%%%%%%%%%%%%%%%%%%%
\subsection{Background}

Benford's Law, named after the physicist and mathematician Frank Benford who observed it in 1938, describes the non-uniform distribution of first digits in many real-world datasets. According to this law (which we define precisely below), the digit $ 1 $ arises as the leading digit approximately $ 30\% $ of the time, $ 2 $ approximately $ 17\% $ of the time, and so on, with larger digits occurring less frequently. This counterintuitive pattern emerges due to the logarithmic nature of the distribution. It can be observed in a wide range of naturally occurring datasets, such as financial reports, census data, scientific constants, and even seemingly unrelated fields like social media statistics. Today, there are numerous applications of Benford's law including in voting fraud detection \cite{Nig}, economics \cite{To, V-BFJ}, geology \cite{NM}, signal processing \cite{PHA}, and the physical sciences \cite{Buc, El, MSPZ, NWR, PTTV, Sha1, Sha2}. See \cite{BeHi, Mil1} for more on the general theory and fields where it is observed.

Given its ubiquity and many applications, it is therefore of interest to study which mathematical processes lead to Benford behavior. In general, it is often true that arithmetic operations (such as sums or products) of random variables yield a random variable that is closer to satisfying Benford's Law \cite{Adh,AS,Bh,JKKKM,Lev1,Lev2,MN1,Rob,Sa,Sc1,Sc2,Sc3,ST}. However, this is not always the case (see for example \cite{BH5}). In certain cases, a central limit theorem law is attainable, where Benfordness follows from the convergence of the distribution of mantissas (see \Cref{sec:benford}) to the uniform distribution.

In 2006, A. Kossovsky studied the distribution of leading digits of chained probability distributions, and conjectured that as the length of the chain increases then the behavior tends to Benford’s law \cite{Ko}. Inspired by this conjecture, Jang, Kang, Kruckman, Kudo and Miller proved in \cite{JKKKM} that if $X_1,\dots,X_m$ are independent continuous random variables with densities $f_1,\dots,f_m$, for any base $B$, for many choices of the densities (more precisely, those that satisfy a certain \textit{Mellin transform condition}, which we describe in detail later), the distribution of the digits of $X_1,\dots,X_m$ converges to Benford's law base $B$ as $m\to\infty$. We prove a more practical criterion for a distribution to satisfy the needed Mellin transform relation.

A nice way to translate such a result into a concrete physical process is by considering the stick fragmentation model, first proposed by Becker et al. \cite{B+} based upon work by the physicist Lemons \cite{Lemons} on partitioning a conserved quantity, which we review in the next subsection. Basically, \cite{B+} asked if a stick of a given length is repeatedly cut into two at random proportions, as the number of levels of this cutting goes to infinity, does the final collection of stick lengths converge to Benford behavior? This formulation indicates why a sequence of products of random variables with increasing lengths is a natural object to consider. Such a process also has discrete variants that resemble the particle decay process in nuclear physics, and thus may be useful for modelling those processes. Some examples of other types of decomposition models include \cite{AF, Bert, Ca, CaVo, IMS, IV, Ka, Kol, Lo, Ol, PvZ, Sl, vZ}.

We study several natural generalizations of the basic stick decomposition model; namely, 
\begin{enumerate}
    \item[i)] continuous fragmentation with random number of parts,
    \item[ii)]continuous fragmentation with probabilistic stopping, and
    \item[iii)] discrete fragmentation with congruence stopping conditions,
\end{enumerate}
which we will discuss in more detail in the subsequent sections.
Indeed, we show that a much larger family of stick fragmentation processes result in strong Benford behavior (defined rigorously in \Cref{sec:benford}).\footnote{There are other generalizations one could study; see \cite{BDMMM, DM} for Benfordness of $d$-dimensional frames of $n$-dimensional boxes.}
In particular, we give an affirmative answer to \cite[Conjecture 8.1(i)]{B+} and prove a result that vastly generalizes their conjecture.  
Notably, we introduce the \emph{aggregated limit}, which is necessary in order to talk about convergence to Benford's law in the latter two models. In order to show that the conditions we require to get convergence to Benford are in fact optimal, we prove non-Benfordness results when our precise conditions are not met.

%%%%%%%%%%%%%%%%%%%%%%%%%%%%%%%%%%%%%%%%%%%%%%%%%%%%%%%%%%%
%%%%%%%%%%%%%%%%%%%%%%%%%%%%%%%%%%%%%%%%%%%%%%%%%%%%%%%%%%%
\subsection{The basic model}
\label{sec:basicmodel}
We recall the following basic stick decomposition model studied in \cite{B+}. 
Start with a stick of length $L$ and fix a continuous probability distribution $\mathcal{D}$ with density function supported on $[0,1]$. Choose $p_1\in [0,1]$ according to $\mathcal{D}$ and break $L$ into $p_1 L$ and $(1-p_1)L$. This is the first \emph{level}. Now for each subsequent level, repeat the same process on every new stick obtained in the previous level, where each breaking involves sampling a new ratio $p_i\in [0,1]$ according to $\mathcal{D}$. Then at the end of the $N$-th level, each resulting stick has length of the form 
\begin{equation}
X_i \ = \ \prod_{n=1}^N p_n,
\end{equation}
where $p_n$ represents the proportion used to cut the ancestor of $X_i$ in the $n$-th level. 

For such a process and its variants, we are interested in whether the final collection of stick lengths $\{X_i\}$ follows \emph{Benford's law} as defined in \Cref{sec:benford}. 
Becker et al. \cite{B+} gave a proof of the Benfordness of the basic process described above, given that the distribution $\mathcal{D}$ satisfies a certain condition involving the convergence of a sum of products of its Mellin transform. This condition was proposed by Jang et al. in \cite[Theorem 1.1]{JKKKM}. We restate it precisely in \Cref{sec:mellinTransformCondition}. There, we also give a sufficient condition for a distribution to satisfy this property. Throughout, we adopt the convention that $\log x$ stands for the natural logarithm of $x$, although the base usually does not play a role unless we explicitly state it.

%%%%%%%%%%%%%%%%%%%%%%%%%%%%%%%%%%%%%%%%%%%%%%%%%%%%%%%%%%%
%%%%%%%%%%%%%%%%%%%%%%%%%%%%%%%%%%%%%%%%%%%%%%%%%%%%%%%%%%%
\subsection{Benford's law}
\label{sec:benford}
Fix a base $B>0$. Any $x>0$ can be written as
\begin{equation}
x\ =\ S_B(x)\cdot B^{k_B(x)}
\end{equation}
where $S_B(x)\in [1,B)$ is the \emph{significand} of $x$ base $B$ and $k_B(x)=\lfloor \log_B(x)\rfloor$ is the \emph{exponent}. The \emph{mantissa} of $x$ is defined to be
\[
M_B(x)\  = \ \log_B(x)-k_B(x).
\]

We have the following standard definition (see for example \cite{MN1}).

\begin{definition}[Benford's law for a sequence]
    A sequence of positive numbers $(a_i)$ is said to be \emph{Benford base $B$} if
    \begin{equation}
        \lim_{I\to\infty} \frac{\#\{i\leq I: 1\leq S_B(a_i)\leq s\}}{I} \ = \ \log_B s
    \end{equation}
    for all $s\in[1,B]$.
\end{definition}

We can also define the notion of Benford behavior for a random variable supported on $(0,\infty)$.

\begin{definition}
\label{def:StrongBenfordForD}
    A probability distribution $\mathcal{D}$, supported on $(0,\infty)$, is said to be \emph{Benford base $B$} if for $X\sim\mathcal{D}$, $M_B(X)$ follows the uniform distribution on $[0,1]$. This is equivalent to saying that
    \begin{equation}
        \prob(1\leq S_B(X)\leq s) \ =\ \log_B s
    \end{equation}
    for all $s\in [1,B]$.
\end{definition}

This is also sometimes referred to as \emph{strong Benfordness}, as opposed to \emph{weak Benfordness}, which only concerns the leading digits of a sequence of numbers. Since we are interested in the limiting behavior of a sequence of finite sets of stick lengths, we give the following precise definition of ``convergence to Benford''.

\begin{definition}
    A sequence of finite collections of positive numbers $(\mathcal{A}_n=\{a_{n,i}\})_n$ is said to \emph{converge to strong Benford behavior (base $B$)}  if 
    \begin{equation}
        \lim_{n\to\infty} \frac{\#\{i: 1\leq S_B(a_{n,i})\leq s\}}{|\mathcal{A}_n|} \ = \ \log_B s
    \end{equation}
    for all $s\in [1,B)$.
\end{definition}

Thus, base 10 the probability of a first digit being $d$ for a sequence that is strong Benford is $\log_{10}(d+1) - \log_{10}(d) = \log_{10}(1 + 1/d)$; in particular the probabilities decrease from about 30.1\% for a leading digit of 1 down to approximately 4.6\% for a 9. While there is thus a tremendous bias towards smaller leading digits, if instead we look at the distribution of the logarithm of the significands (the mantissas) we have the uniform distribution. 

In our case, the collections of random variables representing ending stick lengths are indexed by either the total number of levels $N$, the starting stick length $L$, and/or the number of starting sticks $R$, with these quantities going to infinity in the limit. With an abuse of notation, we always denote the collection of ending stick lengths by $\{X_i\}$ and suppress the parameters $N, L$ and $R$, but the limits will be explicitly stated. Before stating the definition of Benfordness for a sequence of collections of random variables, recall the following notations from \cite{B+}.

For $s\in [0,B)$, we define the indicator function of ``significand at most $s$'' by
\begin{equation}
\varphi_s(x) \ := \ \begin{cases}
    1, & \text{ if the significand of $x$ is at most $s$}\\
    0, & \text{ otherwise.}
\end{cases}
\end{equation}    
Denote the proportion of elements in a set $\{X_i\}$ whose significand is at most $s$ by
\begin{equation}
P(s) \ := \ \frac{\sum_i \varphi_s(X_i)}{\#\{X_i\}}.
\end{equation}

\begin{definition}[\cite{B+}]
\label{def:benfordCriterion}
A sequence of (finite) collections of random variables $(\{X_i\})_n$ is said to \emph{converge to strong Benford behavior (base $B$)} if
%as $*\to\infty$ (where $*$ is replaced by $N$ and $L$ appropriately), it suffices to show that
\begin{enumerate}
    \item \begin{equation}
        \lim_{n\to\infty} \E[P_n(s)] \ = \ \log_B(s)
    \end{equation} and
    \item \begin{equation}
    \lim_{n\to\infty} \Var[P_n(s)] \ = \ 0.
    \end{equation}
\end{enumerate}
\end{definition}

If we break a stick into $k$ pieces using some random process, this can be thought of as sampling from a distribution supported on $[0,1]^{k-1}$ (for the $k-1$ breaking points). From now on, we assume all distributions from which these breaking points are sampled are \emph{good}. This is defined precisely in \Cref{sec:mellinTransformCondition}, but can be thought of as requiring the distribution to ``sufficiently continuous''.

%%%%%%%%%%%%%%%%%%%%%%%%%%%%%%%%%%%%%%%%%%%%%%%%%%%%%%%%%%%%%%%%%%%%%%%%%%%%%%%%%%%%%%%%%%%%%%%%%%%%%%%%%%%%%%%%%%%%%%%%%%%%%%%%%%%%%%%%%%%%%%%%%%%%%%%%%%%%%%%%%%%%%%%%%%%%%%%%%%%%%%%%%%%%%%%%%%%%%%%%%%%%%%%%%%%%%%%%%%%%%%%%%%%%%%%%%%%%%%%%%%%%%%%%%%%%%%%%%%%%%%%%%%%%%%%%%%%%%%%%%%%%%%%%%%%%%%%%%%%%%%%%%%%%%%%%%%%%%%%%%%%%%%%%%%%%%%%%%%%%%%%%%%%%%%%%%%%%%%%%%%%%%%%%%%%%%%%%%%%%%%%%%%%%%%%%%%%%%%%%%%%%%%%%%%%%%%%%%%%%%%%%%%%%

%%%%%%%%%%%%%%%%%%%%%%%%%%%%%%%%%%%%%%%%%%%%%%%%%%%%%%%%%%%
%%%%%%%%%%%%%%%%%%%%%%%%%%%%%%%%%%%%%%%%%%%%%%%%%%%%%%%%%%%
\subsection{Continuous fragmentation with random number of parts} In the basic fragmentation process described in \Cref{sec:basicmodel}, the number of parts each stick breaks into each time is fixed at 2. It is natural to ask whether the same conclusion holds when we allow this number to be randomly chosen as well. Indeed, we prove strong Benfordness of final stick lengths in the following two scenarios:
\begin{enumerate}
    \item the number of parts is chosen independently \emph{at each level}, and is uniform for every stick in that level;
    \item the number of parts is chosen for \emph{each individual stick} within every level.
\end{enumerate}
We now state our results precisely as follows.
Let $G$ be a discrete distribution on $\{1,2,\dots,m\}$ such that $\prob(X=1)<1$ for $X\sim G$. For each $k\in \{1,2,\dots,m\}$, let  $\ef_{k}$ be a finite set of \emph{good} probability distributions with density functions supported on $[0,1]^{k-1}$. 

\begin{theorem} \label{thm: continuous_breaking_levels}
    Start with a stick of length $L$. At each level $i$, independently choose $k \in \{1,2, \dots, m\}$ according to $G$, and break up every stick into $k$ parts by cutting it at the $k-1$ coordinates of a random variable sampled from some distribution in $\mathcal{F}_k$. Then, the distribution of the stick lengths at level $N$ approaches a strong Benford distribution almost surely as $N \to \infty$.
\end{theorem}

\begin{theorem} \label{thm: continuous_breaking_sticks}
    Start with a stick of length $L$. At each level $i$, for each stick in that level, independently choose $k \in \{1,2,\dots, m\}$ according to $G$, and break the stick into $k$ parts by cutting it at the $k-1$ coordinates of a random variable sampled from some distribution in $\mathcal{F}_k$. Then, the distribution of the stick lengths at level $N$ approaches a strong Benford distribution almost surely as $N\to\infty$.
\end{theorem}

%%%%%%%%%%%%%%%%%%%%%%%%%%%%%%%%%%%%%%%%%%%%%%%%%%%%%%%%%%%%%%%%%%%%%%%%%%%%%%%%%%%%%%%%%%%%%%%%%%%%%%%%%%%%%%%%%%%%%%%%%%%%%%%%%%%%%%%%%%%%%%%%%%%%%%%%%%%%%%%%%%%%%%%%%%%%%%%%%%%%%%%%%%%%%%%%%%%%%%%%%%%%%%%%%%%%%%%%%%%%%%%%%%%%%%%%%%%%%%%%%%%%%%%%%%%%%%%%%%%%%%%%%%%%%%%%%%%%%%%%%%%%%%%%%%%%%%%%%%%%%%%%%%%%%%%%%%%%%%%%%%%%%%%%%%%%%%%%%%%%%%%%%%%%%%%%%%%%%%%%%%%%%%%%%%%%%%%%%%%%%%%%%%%%%%%%%%%%%%%%%%%%%%%%%%%%%%%%%%%%%%%%%%%%

\subsection{Continuous fragmentation with probabilistic dying}
We now investigate continuous processes in which the number of parts is fixed throughout, but each new stick \emph{dies} (i.e., stops breaking in subsequent levels) with a certain probability. More specifically, we consider the following fragmentation process. 

Start from $R$ sticks of length $L>0$. Fix a positive integer $k\geq 2$. We call a stick \emph{alive} if it continues to break in the next level and \emph{dead} otherwise. All initial sticks are assumed to be alive and each breaks into $k$ pieces in the first level with the $(k-1)$ breaking points being the coordinates of a random variable chosen from some \emph{good} probability distribution on $[0,1]^{k-1}$. The breaking point random variable of each living stick is independent from each other. After each level, each new stick obtained continues to be \emph{alive} with probability $r$ and \emph{dead} with probability $1-r$. Then we have the following.

\begin{theorem}
\label{conj:contbreakwstopratio}
    When $r=1/k$ and the alive/dead status of each stick is independent, the process ends in finitely many levels with probability 1, and the collection of ending stick lengths almost surely converges to strong Benford behavior as $R \to\infty$.
\end{theorem}
\begin{theorem}
\label{thm:contR>1/k}
    When $r>1/k$, there is positive probability that the process with $R=1$ does not end in finitely many levels.
\end{theorem}

\begin{remark}
    The only remaining case is when $r<1/k$. It is not hard to see that the process ends in finitely many levels in this case. Our numerical simulations strongly suggest that the distribution is non-Benford, but it is an interesting open question as to what distributions result from such processes. 
\end{remark}

We also have the following version of \Cref{conj:contbreakwstopratio} with dependencies between the dead/alive status of different sticks. This serves as a continuous analogue to the discrete processes discussed in the next section in which the dead/alive status of sticks have number theoretic dependencies.
\begin{theorem}[\Cref{conj:contbreakwstopratio} but with dependence]
\label{conj: contbreaks_dependent}
    When $r=1/k$ but the alive/dead status of the children of the same stick are possibly dependent on one another, the collection of stick lengths after $N\geq \log R$ levels almost surely converges to strong Benford behavior as $R \to \infty$.
\end{theorem}

\begin{theorem}
\label{thm: contbreaks_dependent_r_small}
    When $r < 1/k$ and the alive/dead status of the children of the same stick are possibly dependent on one another, the collection of stick lengths does not converge to strong Benford behavior for sufficiently large bases $B$.
\end{theorem}

%%%%%%%%%%%%%%%%%%%%%%%%%%%%%%%%%%%%%%%%%%%%%%%%%%%%%%%%%%%
%%%%%%%%%%%%%%%%%%%%%%%%%%%%%%%%%%%%%%%%%%%%%%%%%%%%%%%%%%%
\subsection{Discrete fragmentation with congruence stopping condition}
Now we turn to the setting of discrete stick fragmentation. For a subset $\mathfrak{S}\subseteq\Z_+$ and a positive integer $L\notin\mathfrak{S}$, define the discrete fragmentation process with starting length $L$ and stopping set $\mathfrak{S}$ as follows.

Start with a stick of integer length $L$. In the first level, it breaks into two new sticks at an integer point chosen according to the uniform distribution on $\{1,\dots,L-1\}$. Now after each level, a new stick becomes \emph{dead} if its length is in $\mathfrak{S}$, the \emph{stopping set}, and continues to be \emph{alive} otherwise. The starting stick is assumed to be alive. When a stick is \emph{dead}, it no longer breaks in the subsequent levels. Note that the \emph{stopping condition}, namely the condition for a stick to become dead,
is described by the subset $\mathfrak{S}\subseteq \Z_+$, called the \emph{stopping set}.
In the next level, each of the living sticks continues to break into two following the discrete uniform distribution. The process ends when all new sticks meet the stopping condition, i.e., when all sticks at the end of a level are dead. We are interested in whether the final collection of stick lengths converge to the Benford distribution as we take the limit $L\to\infty$.

In \cite{B+}, the authors showed the Benfordness of a discrete process in which the breaking only continues on one side of the stick, with stopping condition being length equal to $1$. Here we study processes with more general stopping conditions defined by congruence classes. Our results are the following.

\begin{theorem} \label{thm:evenStopBreaking}
    Start with a stick of \emph{odd} integer length $L$. Let the stopping set be $\mathfrak{S}=\{1\}\cup \{2m: m\in\Z_+\}$. Namely, a stick dies whenever its length is $1$ or even. Then the distribution of lengths of all dead sticks at the end approaches strong Benfordness as $L \to \infty$.
\end{theorem}

% The experimental results inspired the following generalization of \Cref{thm:evenStopBreaking}.

\begin{theorem}
\label{thm:discreteStopAtHalfResidues}
    Fix an even modulus $n \geq 2$ and a subset $S\subset\{0,\dots,n-1\}$ of size $n/2$ representing the residue classes. Let the stopping set be
    \begin{equation}
        \mathfrak{S} \ := \ \{1\}\cup\{m\in\Z_+ :m = qn+r,\  r\in S, q\in \Z\}.
    \end{equation}
    If we start with $R$ identical sticks of positive integer length $L\notin\mathfrak{S}$, then the collection of ending stick lengths converges to strong Benford behavior given that $R>(\log L)^3$ as $L \to\infty$.
    % Consider the following discrete fragmentation process.
    % Fix some positive integer $L\notin\mathfrak{S}$.
    % Begin with $R$ identical sticks of length $L$. 
    % We call a stick \emph{alive} if it continues to break in the next level and \emph{dead} otherwise. All initial sticks are assumed to be alive.
    % At each level, for each living stick of length $\ell > 1$, choose a random integer from $\{1,\dots,\ell-1\}$ uniformly and cut the stick into two pieces at that point. Each new stick becomes dead if it is in the stopping set and stays alive otherwise. Repeat this process on the new set of living sticks until all sticks are dead. The collection of ending stick lengths converges to strong Benford behavior if $R>(\log L)^3$ as $L \to\infty$.
\end{theorem}

We also prove non-Benfordness results when $|S|\neq n/2$ and make more specific conjectures in \Cref{sec:discreteNonBenford}. In \Cref{sec:discreteGeneralParts}, we discuss generalizing the process to one where the number of parts each stick is broken into is a chosen integer $k\geq 2$.

\begin{remark}
    Given that our results involve random decomposition of residue classes modulo a given integer, they could be of number theoretic interest. One can ask further whether similar results hold when $\mathfrak{S}$ is given by other subsets of $\Z_+$ that arise in number theory, for example, the set of quadratic residues modulo a given integer, the set of primes or practical numbers, etc. 
\end{remark}

We note that the only property of $\mathfrak{S}$ that is fundamentally necessary in our above work involving discrete breaking processes seems to be the density of the set in the natural numbers. As result, we have the following conjecture:

\begin{conjecture}
    Let $\mathfrak{S}$ be such that the limit below exists and let
    \begin{equation}
        r = \lim_{n \to \infty} \frac{|\{[1,n] \cap \mathfrak{S}\}|}{n}.
    \end{equation}
    Moreover, assume that $r > 0$. Then, we have that set of dead stick lengths approaches Benford behavior if and only if $r = 1/2$.
\end{conjecture}

In the next section, we briefly review the Mellin transform condition and prove a practical sufficient criterion (\Cref{thm:Mellin_transform_condition_continuity}) for it to hold. A family of examples satisfying the Mellin transform condition that follow from that criterion is given in \Cref{ex:condition_for_goodness}. 
Subsequent sections include the proofs of our main results and further discussions on our conjectures.

%%%%%%%%%%%%%%%%%%%%%%%%%%%%%%%%%%%%%%%%%%%%%%%%%%%%%%%%%%%%%%%%%%%%%%%%%%%%%%%%%%%%%%%%%%%%%%%%%%%%%%%%%%%%%%%%%%%%%%
%%%%%%%%%%%%%%%%%%%%%%%%%%%%%%%%%%%%%%%%%%%%%%%%%%%%%%%%%%%%%%%%%%%%%%%%%%%%%%%%%%%%%%%%%%%%%%%%%%%%%%%%%%%%%%%%%%%%%%
\section{Preliminaries}

%%%%%%%%%%%%%%%%%%%%%%%%%%%%%%%%%%%%%%%%%%%%%%%%%%%%%%%%%%%
%%%%%%%%%%%%%%%%%%%%%%%%%%%%%%%%%%%%%%%%%%%%%%%%%%%%%%%%%%%
\subsection{Mellin transform condition}
\label{sec:mellinTransformCondition}
Becker et al. \cite{B+} gave a proof of the Benfordness of the basic process described at the beginning of the section, given that the distribution $\mathcal{D}$ satisfy a certain condition involving the convergence of a sum of products of its Mellin transform. This condition was proposed by Jang et al. in \cite[Theorem 1.1]{JKKKM}. We restate it precisely as follows.

For a continuous real-valued function $f:[0,\infty)\to\R$, let $\mathcal{M}f$ denote its \emph{Mellin transform} defined by
\begin{equation}
    \mathcal{M}f(s)\ =\ \int_{0}^\infty f(x)x^s\frac{dx}{x}. 
\end{equation}

Let $\ef = \{\mathcal{D}_j\}_{j\in I}$ be a family of probability distributions with associated density functions $f_j$ supported on $[0,\infty)$ and $p: \Z_+ \to I$. We say that $\ef$ satisfies the \emph{Mellin transform condition} if the following holds and the convergence is uniform over all choices of $p$:
\begin{equation}
    \label{eq:mellincondition}\lim_{n\to\infty}\sum_{\underset{\ell\neq 0}{\ell = -\infty}}^{\infty}\prod_{m=1}^n\mathcal{M}f_{\mathcal{D}_{p(m)}}\left(1-\frac{2\pi i\ell}{\log B}\right)\ =\ 0.
\end{equation}
The following corollary of \cite[Theorem 1.1 \& Lemma 1.2]{JKKKM} relating the Mellin transform property to Benford behavior will be used repeated in our proofs of Benfordness results, so we restate it here for ease of reference.
\begin{theorem}[{\cite[Theorem 1.1]{JKKKM}}]
\label{thm:productConvergeToBenford}
    Let $\ef = \{\mathcal{D}_j\}_{j\in I}$ be a family of probability distributions with associated density functions $f_j$ supported on $[0,\infty)$ satisfying the Mellin transform property and $p: \Z_+ \to I$. Let $X_1\sim\mathcal{D}_{p(1)}$. For all $i\geq 2$, let $X_i$ be a random variable with probability density function given by
    \begin{equation}
    \theta^{-1}f_{\mathcal{D}_{p(i)}}(x/\theta)
    \end{equation}
    where $\theta$ is the value of the previous random variable $X_{i-1}$. Then if $Y_n = \log_B X_n$, we have
    \begin{equation}
    \begin{split}
    |&\prob(Y_n\mod 1\in [a,b])-(b-a)|
    \\
    &\leq\ (b-a)\cdot\left|\lim_{n\to\infty}\sum_{\underset{\ell\neq 0}{\ell = -\infty}}^{\infty}\prod_{m=1}^n\mathcal{M}f_{\mathcal{D}_{p(m)}}\left(1-\frac{2\pi i\ell}{\log B}\right)\right|.
    \end{split}
    \end{equation}
    In particular, the limiting distribution as $n\to\infty$ of $X_n$ is Benford base $B$.
\end{theorem}

Note that from the way $X_n$ is defined, it precisely models the product of $n$ random variables, each distributed according to $\mathcal{D}_{p(i)}$ for $1\leq i\leq n$.
The following gives a general condition on $\ef$ for it to satisfy the Mellin transform condition. A weaker version of the result is briefly discussed in \cite{JKKKM}.

\begin{theorem} \label{thm:Mellin_transform_condition_continuity}
    $\ef$ satisfies the Mellin transform condition if it is finite and all $f_j\in\ef$ are $\alpha_j$-H\"older continuous $(0 < \alpha_j \le 1)$ and supported only on $[0,1]$. In particular, for such an $\ef$, a sequence of products of random variables distributed according to some sequence of the $f_{j}\in\ef$ approaches Benford behavior, and the rate of this convergence is uniform over all such sequences.
\end{theorem}

Consider a probability distribution $\mathcal{D}$ on $\mathbb{R}^m$ that is supported on $[0,1]^m$ with cumulative distribution function $F$. For $X\sim\mathcal{D}$, Let $\mathrm{rk}_i(X)$ denote its $i$th smallest coordinate, where $1\leq i \leq m$. Let $\mathrm{rk}_0(X) = 0$ and $\mathrm{rk}_{m+1}(X) = 1$. Then, we say that $\mathcal{D}$ is \emph{good} if 
\begin{equation}
    Y_i\ =\ \mathrm{rk}_{i+1}(X) - \mathrm{rk}_i(X)
\end{equation}
has H\"older continuous density for all $0 \le i \le m$. In other words, if $X$ represents the cut points of a stick, then we require the distances between adjacent ones to have H\"older continuous densities. This definition is necessary for exploring stick breaking, in which we must choose multiple breaking points of a stick from a distribution and then consider distributions of ratios between the lengths of children and their parents. That is, if such a distribution is \emph{good}, then \Cref{thm:Mellin_transform_condition_continuity} applies. In fact, many distributions of interest are \emph{good}. For instance, we have the following family of examples.

\begin{example} \label{ex:condition_for_goodness}
    Suppose that $\mathcal{D}$ is the product of $m$ independent $1$-dimensional distributions $\mathcal{D}_i$ with densities $f_i$ and cumulative densities $F_i$. If the $f_i$ are H\"older continuous, then $\mathcal{D}$ is good.
    \begin{proof}
Let $Y_i = \mathrm{rk}_{i+1}(X) - \mathrm{rk}_i(X)$ for some $X \sim D$. Assume that $1 \le i < m$. Then
\begin{equation}
    1 - F_{Y_i}(c) \ = \ \sum_{j=1}^m \sum_{\substack{S \subseteq [m] \setminus \{j\} \\ |S| = i-1}} \int_0^{1} f(x) \prod_{l \in S} F_l(x) \prod_{l \not\in S, \ l \neq j} (1 - F_l(x+c)) \ dx
\end{equation}
where we sum over the possible $X_j$ that correspond to $\mathrm{rk}_i(X)$ and the possible sets $S$ of the other variables that are less than $X_j$. By continuity, we can differentiate with respect to $c$ and move the differentiation inside of the integral to obtain
\begin{equation}
    f_{Y_i}(c) \ = \ \sum_{j=1}^m \sum_{\substack{S \subseteq [m] \setminus \{j\} \\ |S| = i-1}} \int_0^{1} f(x) \prod_{l \in S} F_l(x) \sum_{l' \not\in S, \ l' \neq j} f_l(x+c) \prod_{l \not\in S, \ l \neq j,l'} (1 - F_l(x+c)) \ dx.
\end{equation}
Now, the $F_i$ are continuously differentiable, so the are also H\"older continuous. We can then take the minimal exponent $\alpha$ among the the $f_i$ to obtain that $f_{Y_i}$ is $\alpha$-H\"older continuous since sums and products of $\alpha$-H\"older continuous functions are $\alpha$-H\"older continuous. We can similarly show that $f_{Y_i}$ is $\alpha$-H\"older continuous when $i = 0, m$.
    \end{proof}
\end{example}

From now on, we assume all distributions from which breaking points are sampled are \emph{good}.

%%%%%%%%%%%%%%%%%%%%%%%%%%%%%%%%%%%%%%%%%%%%%%%%%%%%%%%%%%%
%%%%%%%%%%%%%%%%%%%%%%%%%%%%%%%%%%%%%%%%%%%%%%%%%%%%%%%%%%%
\subsection{Proof of continuity criterion for the Mellin transform condition }
\begin{proof}[Proof of \Cref{thm:Mellin_transform_condition_continuity}]
Note that, for fixed $j \in I$,
\begin{align}
    \mathcal{M}f_j\left(1-\frac{2\pi i\ell}{\log B}\right) &\ = \nonumber \ \int_{0}^\infty f_j(x) x^{-\frac{2\pi i\ell}{\log B}} dx \\ \nonumber
    &\ = \ \int_0^\infty f_j(e^{\log x}) e^{\log x} e^{-\frac{2\pi i\ell}{\log B} \log x} \frac{dx}{x} \\ \nonumber
    &\ = \ \int_{-\infty}^\infty g_j(y) e^{-\frac{2\pi i\ell}{\log B} y} dy \\ 
    &\ = \ \widehat{g_j}\left( \frac{\ell}{\log B} \right),
\end{align}
where $g_j(y) = f_j(e^{y})e^y$. Moreover, $\|g_j\|_1 = \|f_j\|_1 = 1$, so the Riemann-Lebesgue lemma applies and says that 
\begin{equation}
    \mathcal{M}f_j\left(1-\frac{2\pi i\ell}{\log B}\right) \ \to \ 0
\end{equation}
as $\ell \to \infty$. Also, for $\ell \neq 0$,
\begin{equation}
    \left| \widehat{g_j}\left( \frac{\ell}{\log B} \right) \right| \ \le \ \|g_j\|_1 \ = \ 1.
\end{equation}
We do not have equality in the above since it follows from triangle inequality and the integrand does not always have the same complex argument (since $g_j$ is continuous). Thus, if we take
\begin{equation}
    h(\ell) \ = \ \max_j \left| \mathcal{M}f_j\left(1-\frac{2\pi i\ell}{\log B}\right) \right|,
\end{equation}
we have that $h(\ell) < 1$ for $\ell \neq 0$ and also $h(\ell) \to 0$ as $\ell \to \infty$. We now investigate the rate of this convergence. We begin by mimicking the proof of the Riemann-Lebesgue lemma.
For any $f:\R\to\C$ continuous and compactly supported, using the substitution $x\mapsto x+\frac{\pi}{\xi}$ for $\xi\neq 0$, we have
\begin{equation}
     {\hat {f}}(\xi )\ =\ \int _{\mathbb {R} }f(x) {e} ^{- {i} x\xi } {d} x\ =\ \int _{\mathbb {R} }f\left(x+{\frac {\pi }{\xi }}\right) {e} ^{- {i} x\xi } {e} ^{- {i} \pi } {d} x\ =\ -\int _{\mathbb {R} }f\left(x+{\frac {\pi }{\xi }}\right) {e} ^{- {i} x\xi } {d} x.
\end{equation}
Taking the average, we get
\begin{equation}
     |{\hat {f}}(\xi )|\ \leq\ {\frac {1}{2}}\int _{\mathbb {R} }\left|f(x)-f\left(x+{\frac {\pi }{\xi }}\right)\right| {d} x.
\end{equation}
Apply this to $f=\widehat{g_j}$ and $\xi =  \frac{\ell}{\log B} $,
\begin{align}
    \left|\widehat{g_j}\left( \frac{\ell}{\log B} \right)\right| \nonumber
    &\ \leq\ \frac{1}{2}\int_{\R}\left|g_j(x)-g_j\left(x+\frac{\pi\log B}{\ell}\right)\right|dx\\ \nonumber
    &\ \leq\ \frac{1}{2}\int_{\R}\left|f_j(e^x)e^x-f_j(e^{x+\frac{\pi\log B}{\ell}})e^{x+\frac{\pi\log B}{\ell}}\right|dx\\ \nonumber
    &\ \leq\ \frac{1}{2}\int_{0}^1\left|f_j(u)-cf_j(cu)\right|du\\
    &\ \leq\ \frac{1}{2}\sup_{[0,1]}(|f_j(u)-f_j(cu)|+|f_j(cu)-cf_j(cu)|),      
\end{align}
where $c=e^{\frac{\pi\log B}{\ell}}$ and we used the fact that $f$ is only supported on $[0,1]$ (this can be easily changed to any compact interval, but for the purpose of this paper all the distributions we consider satisfy this condition).
From the assumption that $f_j$ is $\alpha_j$-H\"older continuous, there exists a constant $\mu\geq 0$ such that 
\begin{equation}
    |f_j(u)-f_j(cu)| \ \leq \ \mu|(1-c)u|^\alpha \ \leq \ \mu|1-c|^\alpha
\end{equation}
and
\begin{equation}
    |f_j(cu)-cf_j(cu)|\ \leq \ (1-c)|f_j(cu)| \ \leq \ |1-c|M
\end{equation}
for all $u\in [0,1]$,
where $M>0$ is an upper bound for $f$.
Now 
\begin{equation}
    c\ =\ 1 + \frac{\pi\log B}{\ell} + o(1/\ell),
\end{equation}
 so
\begin{equation}
|1 - c|\ =\  \frac{\pi\log B}{\ell} + o(1/\ell).
\end{equation}
We may assume $0<\alpha\leq 1$, so that $|1-c|^\alpha$ dominates. There exists some $L$ large enough so that the sum $\sum_{|\ell|\geq L}\ell^{-n\alpha}\to 0$ as $n\to\infty$. 
By the pigeonhole principle, there exist a $j\in I$ such that $|p^{-1}(\{j\})|=\infty$. Then we have
\begin{align}
    \sum_{\underset{\ell\neq 0}{\ell = -\infty}}^{\infty}\prod_{m=1}^n\mathcal{M}f_{\mathcal{D}_{p(m)}}\left(1-\frac{2\pi i\ell}{\log B}\right) \nonumber
    &\ \leq\ 
    \sum_{\underset{\ell\neq 0}{\ell = -\infty}}^{\infty}\prod_{\underset{p(m)=j}{m=1}}^n\mathcal{M}f_j\left(1-\frac{2\pi i\ell}{\log B}\right)\\
    &\ \leq\ 
    \sum_{\underset{\ell\neq 0}{\ell = -\infty}}^{\infty}\mathcal{M}f_j\left(1-\frac{2\pi i\ell}{\log B}\right)^n
    \ \to \ 0
\end{align}
as $n\to\infty$. This proves \eqref{eq:mellincondition}. To see that the convergence is uniform over all $p$, note that by the pigeonhole principle, for any choice of $p$ and any positive integer $N$ there exists a $j\in I$ such that $|p^{-1}(\{j\})\cap\{1,\dots,N\}|\ \geq\ N/|I|$. Therefore for any $\epsilon>0$, it suffices to take the maximum among all the $N$'s needed for each $j\in I$ so that 
\begin{equation}
    \sum_{\underset{\ell\neq 0}{\ell = -\infty}}^{\infty}\mathcal{M}f_j\left(1-\frac{2\pi i\ell}{\log B}\right)^{N/|I|} \ < \ \epsilon.
\end{equation}    
\end{proof}

%%%%%%%%%%%%%%%%%%%%%%%%%%%%%%%%%%%%%%%%%%%%%%%%%%%%%%%%%%%%%%%%%%%%%%%%%%%%%%%%%%%%%%%%%%%%%%%%%%%%%%%%%%%%%%%%%%%%%%
%%%%%%%%%%%%%%%%%%%%%%%%%%%%%%%%%%%%%%%%%%%%%%%%%%%%%%%%%%%%%%%%%%%%%%%%%%%%%%%%%%%%%%%%%%%%%%%%%%%%%%%%%%%%%%%%%%%%%%
\section{Continuous fragmentation with random number of parts}
In this section, we present proofs of the results on the generalizations of the basic model by allowing the number of parts to be randomly chosen. For both \Cref{thm: continuous_breaking_levels} and \Cref{thm: continuous_breaking_sticks}, we show that 
\begin{equation}
        \lim_{N\to\infty} \E[P_N(s)] \ = \ \log_B(s)
\end{equation} and
\begin{equation}
    \label{eq:variance}
    \lim_{N\to\infty} \Var[P_N(s)] \ = \ 0
\end{equation}
where $N$ is the number of levels.

%%%%%%%%%%%%%%%%%%%%%%%%%%%%%%%%%%%%%%%%%%%%%%%%%%%%%%%%%%%
%%%%%%%%%%%%%%%%%%%%%%%%%%%%%%%%%%%%%%%%%%%%%%%%%%%%%%%%%%%
\subsection{Proof of Theorem \ref{thm: continuous_breaking_levels}}
We first assume that $\prob(k=1) = 0$, i.e., that $k$ is chosen from $\{2,\dots,m\}$. Then, the problem is solved similarly as in Section 2 of \cite{B+}. Let $Y_i$ be the number of sticks that each stick is cut into at level $i$. For simplicity, we assume $i$ always runs through $\{1,\dots, N\}$ and omit the range in the equations below. Note that
\begin{equation}
\E(P_N(s)) \ = \ \sum_{(y_i) \in \{2,\dots,m\}^N} \E(P_N(s) | Y_i = y_i \ \forall i)\prob(Y_i = y_i \ \forall i),
\end{equation}
so it suffices to show that 
\begin{equation}\E(P_N(s) | Y_i = y_i \ \forall i)\ =\ \E(\varphi_s(p_1p_2 \cdots p_N))\ \to\ \log_B(s)\end{equation} 
at a uniform rate for any choice of splittings $(Y_i)=(y_i)$ as $N \to \infty$, where $p_i\sim \mathcal{D}_i$ with $\mathcal{D}_i\in\ef_{y_i}$. This follows immediately from \Cref{thm:Mellin_transform_condition_continuity}, given our assumption that the distributions in $\ef_k$ are \emph{good} and $|\ef_k|<\infty$ for all $k$. 
The variance can also be bounded independently of the values of the variables $Y_i$. Let $M = y_1y_2 \cdots y_N$ be the number of sticks after $N$ levels. We have
\begin{equation}
    P_N(s) \ = \ \frac{1}{M}\sum_{i=1}^M \varphi_s(X_i)
\end{equation}
so that
\begin{align}
    \Var(P_N(s)) \ &= \ \E[P_N(s)^2] - \E[P_N(s)]^2 \nonumber \\
    &= \ \E \left[ \frac{1}{M^2}\sum_{i=1}^M \varphi_s(X_i)^2 + \frac{1}{M^2}\sum_{\substack{i\neq j \nonumber \\ 1 \le i,j \le M}} \varphi_s(X_i)\varphi_s(X_j) \right] - \E[P_N(s)]^2 \\ 
    &= \ \frac{1}{M^2}\E[P_N(s)] + \frac{1}{M^2}\sum_{\substack{i\neq j \\ 1 \le i,j \le M}} \E[\varphi_s(X_i)\varphi_s(X_j)] - \E[P_N(s)]^2. \label{1_var_expansion}
\end{align}
It therefore suffices to show that 
\begin{equation}
    \frac{1}{M^2}\sum_{\substack{i\neq j \\ 1 \le i,j \le M}} \E[\varphi_s(X_i)\varphi_s(X_j)] \ \to \ \log^2_B(s)
\end{equation}
as $N \to \infty$ (since the first term in \eqref{1_var_expansion} goes to $0$ and $\E[P_N(s)]^2 \to \log^2_B(s)$ by the above). By the same reasoning as in the proof of  \cite[Theorem 1.5]{B+}, if
\begin{equation*}
X_i \ = \ Lp_1p_2 \dots p_{N-n-1}p_{N-n} \dots p_N,
\end{equation*}
\begin{equation}
X_j \ = \ Lp_1p_2 \dots p_{N-n-1}q_{N-n} \dots q_N
\end{equation}
where exactly the last $n$ breaks are independent, then 
\begin{equation}
\left|\E[\varphi_s(X_i)\varphi_s(X_j)] - \log_B^2(s) \right| \ \le \ f(n)
\end{equation}
where $f(n)$ depends only on $n$ and goes to 0 as $n\to \infty$. For a fixed $X_i$, there exists at most $\prod_{i=N-n}^N y_i$ (in fact, precisely $(y_{N-n}-1)\prod_{i=N-n+1}^N y_i$) sticks $X_j$ that involve $n$ independent breaks from $X_i$ (i.e., have a shared ancestor $n+1$ breaks ago). Thus,
\begin{align}
    \left| \frac{1}{M^2}\sum_{\substack{i\neq j \\ 1 \le i,j \le M}} \E[\varphi_s(X_i)\varphi_s(X_j)] - \log_B^2(s) \right| \ \nonumber &\le \ \frac{1}{M^2} \sum_{i = 1}^M \sum_{n=0}^{N-1} \prod_{i=N-n}^N y_i f(n)  \\ \nonumber
    &= \ \sum_{n=0}^{N-1} f(n) \prod_{i = 1}^{N-n-1} \frac{1}{y_i}  \\
    &\le \ \sum_{n=0}^{N-1} 2^{n+1-N} f(n). 
\end{align}
Since $\sum_{n=0}^{N-1} 2^{n+1-N}\leq 2$, the last sum is a weighted average of the $f(n)$'s where the terms with larger $n$ has higher weight. Therefore the sum goes to $0$ as $N \to \infty$ since $f(n)\to 0$ as $n\to\infty$

To cover the case when $Y_i = 1$ is allowed with positive probability strictly less than 1, simply note that almost surely the number of levels involving the stick breaking into greater than $1$ piece increases without bound as $N \to \infty$.

%%%%%%%%%%%%%%%%%%%%%%%%%%%%%%%%%%%%%%%%%%%%%%%%%%%%%%%%%%%
%%%%%%%%%%%%%%%%%%%%%%%%%%%%%%%%%%%%%%%%%%%%%%%%%%%%%%%%%%%
\subsection{Proof of Theorem \ref{thm: continuous_breaking_sticks}}
    The proof of $\lim_{N\to\infty}\E(P_N(s))=\log_B(s)$ is exactly the same as in the previous case. To show $\lim_{N\to\infty}\Var(P_N(s))=0$, it suffices to show that 
\begin{equation}
\frac{1}{M^2}\sum_{
%\substack{i\neq j \\
1 \le i,j \le M}
%} 
\E[\varphi_s(X_i)\varphi_s(X_j)] -
\log^2_B(s)
\ = \
\frac{1}{M^2}\sum_{
%\substack{i\neq j \\ 
1 \le i,j \le M}
%} 
\Big(\E[\varphi_s(X_i)\varphi_s(X_j)] -
\log^2_B(s)\Big)
\end{equation}
goes to zero.
Recall that if $X_i, X_j$ have exactly $n$ independent breaks, then 
\begin{equation}
\left|\E[\varphi_s(X_i)\varphi_s(X_j)] - \log_B^2(s) \right| \ \le \ f(n)
\end{equation}
where $f(n)$ depends only on $n$ and goes to 0 as $n\to \infty$.
So we have
\begin{align}
&\frac{1}{M^2}\sum_{
1 \le i,j \le M}
\Big(\E[\varphi_s(X_i)\varphi_s(X_j)] -
\log^2_B(s)\Big) \nonumber \\
\ \leq \
&\sum_{n=0}^{N-1}
f(n)\cdot\frac{\#\{(i,j): \text{$(X_i,X_j)$ have exactly $n$ independent breaks}\}}{M^2}.
\label{var_bound_1}
\end{align}
Let 
\begin{equation}
A_n \ = \ \#\{(i,j): \text{$(X_i,X_j)$ have exactly $n$ independent breaks}\}.
\end{equation}
Then 
\begin{equation}
\sum_{n=0}^{N-1} A_n \ = \ M^2.
\end{equation}
So the RHS of \eqref{var_bound_1}
is a weighted average of the $f(n)$'s. 
We show that with probability tending to 1, 
\begin{equation}
    \frac{1}{M^2}\sum_{n=0}^{\log\log N} A_n\ \to \ 0
    \label{eq:small_n}
\end{equation}
 as $N\to\infty$.  Intuitively, we would expect the number of pairs of $(X_i,X_j)$ having at most $\log\log N$ independent breaks (i.e., having high dependence) to only make up a very small proportion of the $M^2$ pairs in total.
% We have that
% $$
% A_1 + A_2 + \dots + A_{\lfloor \log \log N \rfloor} = \#\{(i,j): (X_i, X_j) \text{ have at most $\log\log N$ independent breaks}\}.
% $$
If so, 
\begin{equation}
\sum_{n=0}^{\lfloor\log\log N\rfloor} f(n) \frac{A_n}{M^2} \ \to \ 0
\end{equation}
as $N\to\infty$ since $f(n)$ is bounded.
On the other hand, 
\begin{equation}
\sum_{n=\log\log N+1}^{N} f(n) \frac{A_n}{M^2} \ \leq \ f(\log\log N+1) \ \to \ 0
\end{equation}
as $N\to\infty$, therefore we would have that the RHS of \eqref{var_bound_1} tends to 0 as $N\to\infty$, as desired.

Now we show $\eqref{eq:small_n}$. 
Let $\alpha_1, \alpha_2, \dots, \alpha_r$ be the sticks at the $(N - \lfloor \log \log N \rfloor)$-th level. Let $a_i$ be the number of sticks that come from $\alpha_i$ at the end of all $N$ levels. Then, any two $(X_i, X_j)$ having at most $\log\log N$ independent breaks share some $\alpha_l$ as an ancestor. Thus, the number of such pairs is 
\begin{equation}
a_1^2 + a_2^2  + \cdots + a_r^2.
\end{equation}
Let $Z_1:=a_1^2 + a_2^2 + a_3^2 + \dots a_r^2$ and $Z_2:=(a_1 + a_2 + \dots + a_r)^2=M^2$.
Then 
\begin{equation}
    \frac{1}{M^2}\sum_{n=0}^{\lfloor \log \log N \rfloor} {A_n} \ = \ \frac{a_1^2 + a_2^2 + \dots + a_r^2}{M^2} \ = \ \frac{Z_1}{Z_2}.
\label{eq:z1/z2}
\end{equation}

We claim that $r > \sqrt{N}/2$ with high probability. Consider a sequence of $\lceil \sqrt{N}\rceil$ levels and let the number of sticks in the highest level be $\ell$. The probability that the number of sticks does not increase at all throughout the levels, i.e., that all levels have $\ell$ sticks is at most $p_1^{\ell\sqrt{N}}$. Thus, we can look at the blocks of $\lceil \sqrt{N}\rceil$ levels and deduce that the probability of increasing the number of sticks at least once is at least
\begin{equation}
    \prod_{\ell=1}^{\sqrt{N}} \left( 1 - p_1^{\ell\sqrt{N}} \right) \ \ge \ 1 - \sum_{\ell = 1}^\infty p_1^{\ell\sqrt{N}} \ = \ 1 - \frac{p_1^{\sqrt{N}}}{1 - p_1^{\sqrt{N}}}.
\end{equation}
This probability thus approaches $1$ as $N \to \infty$. This shows that with probability approaching 1, the total number of sticks increases by at least 1 within each block, so there will be more than $\sqrt{N}/2$ sticks at the $(N - \lfloor \log \log N \rfloor)$-th level.

Now, fix $r$, i.e., condition on the value of $r$ and assume $r > \sqrt{N}/2$. The $a_i$ are independent and identically distributed. Let $\mu_j = \E[a_i^j]$ and $\sigma^2 = \Var[a_i]$. We have
\begin{equation}
\E[Z_1] \ = \ r\mu_2, \quad \Var[Z_1] \ = \ r\Var[a_i^2] \ =\ r\mu_4 - r\mu_2^2
\end{equation}
and
\begin{align}
    \E[Z_2] \ &= \ r^2\mu^2 + r\sigma^2 \ = \ r\mu_2 + r(r-1)\mu_1^2, \\
    \Var[Z_2] \ &= \nonumber \ \E[(a_1 + a_2 + \dots + a_r)^4] - \E[(a_1 + a_2 + \dots + a_r)^2]^2 \\ \nonumber
    &= \nonumber \ r\mu_4 + 4r(r-1)\mu_1\mu_3 + 4r(r-1)\mu_2^2 + 6r(r-1)(r-2)\mu_1^2\mu_2 \\ \nonumber
    &\quad \ + r(r-1)(r-2)(r-3)\mu_1^4 - (r\mu_2 + r(r-1)\mu_1^2)^2 \\ \nonumber
    &= \ \mu_1^4 r^4 + (6\mu_1^2\mu_2 - 6\mu_1^4)r^3 + (4\mu_1\mu_3 + 4\mu_2^2 - 18\mu_1^2\mu_2)r^2 \nonumber \\
    &\quad \ + (\mu_4 - 4\mu_1\mu_3 - 4\mu_2^2 + 12 \mu_1^2\mu_2 - 6\mu_1^4)r \nonumber \\
    &\quad \ - [\mu_1^4r^4 + (\mu_1^2\mu_2 - \mu_1^4)r^3 + (\mu_2^2 - 2\mu_1^2\mu_2 + \mu_1^4)r^2] \nonumber \\
    & \le \ 5\mu_1^2\mu_2 r^3 + (4\mu_1\mu_3 + 3\mu_2^2)r^2 + (\mu_4 + 12\mu_1^2\mu_2) r.
\end{align}
Now, we have that
\begin{equation} \label{moment_bound}
    \mu_j \ \le \ \max(a_i)^j \ \le \ m^{j\log \log N} \ = \ (\log N)^{j \log m}
\end{equation}
so that
\begin{equation} \label{large_variance_bound}
    \Var[Z_2] \ \le \ 5(\log N)^{C}r^3 + 7(\log N)^{C}r^2 + 13(\log N)^{C}r \ \le \ 25(\log N)^{C}r^3
\end{equation}
where $C = 4 \log m$. Moreover, we may use \eqref{moment_bound} to obtain
\begin{equation} \label{small_variance_bound}
    \Var[Z_1] \ \le \ r\mu_4 \ \le \ (\log N)^C r.
\end{equation}
We also have the following bounds (by trivially lower bounding moments by $1$).
\begin{equation} \label{expectation_bounds}
    \E[Z_1] \ \ge \ r, \quad \E[Z_2] \ \ge \ r^2.
\end{equation}
% \begin{theorem}[Chebyshev's inequality]
%     Let $X$ be a (integrable) random variable with finite non-zero variance $\sigma^2$ and expected value $\mu$.Then for any real number $k>0$,
%     \begin{equation}
%     \prob(|X-\mu|>k) \ \leq \ \frac{\sigma^2}{k^2}.
%     \end{equation}
% \end{theorem}
Now apply Chebyshev's inequality to both $Z_1$ and $Z_2$ to get
% taking $k = \frac{1}{2}\E[Z_i]$.
\begin{equation}
\prob(Z_1 > \frac{3}{2}\E[Z_1]) \ \leq \  
\frac{\Var[Z_1]}{\frac{1}{4}\E[Z_1]^2}
 \ \leq \ \frac{4(\log N)^C r}{r^2\mu_2^2}
 \ \leq \ \frac{4(\log N)^C}{r}
\end{equation}
and
\begin{equation}
\prob(Z_2 < \frac{1}{2}\E[Z_2]) \ \leq \ 
\frac{\Var[Z_2]}{\frac{1}{4}\E[Z_2]^2}
 \ \leq \ \frac{100(\log N)^C r^3}{(r\mu_2+r(r-1)\mu_1^2)^2}
 \ \leq \ \frac{100(\log N)^C r^3}{r^4}
 \ \leq \ \frac{100(\log N)^C}{r}.
\end{equation}
Note that in simplifying the above two expressions we used the trivial bounds $\mu_j\geq 1$ for the denominators.
Recall that with probability going to 1, $r\geq \sqrt{N}/2$.
Under this assumption, as $N\to\infty$,
\begin{align}
    \prob\left(\frac{Z_1}{Z_2}< 
\frac{\frac{3}{2}\E[Z_1]}{\frac{1}{2}\E[Z_2]}\right) \nonumber &\ \geq \ 
1-\prob\left(Z_1 > \frac{3}{2}\E[Z_1]\right)-\prob\left(Z_2 < \frac{1}{2}\E[Z_2]\right) \\
&\ = \ 1 -  \frac{4(\log N)^C}{\sqrt{N}/2} - \frac{100(\log N)^C}{\sqrt{N}/2} \ \to \ 1.
\end{align} 
Since 
\begin{equation}
\frac{\frac{3}{2}\E[Z_1]}{\frac{1}{2}\E[Z_2]} \ = \ \frac{3r\mu_2}{r\mu_2+r(r-1)\mu_1^2}
\ \leq \ \frac{3(\log N)^{j\log m}}{r}
\ \leq \ \frac{3(\log N)^{j\log m}}{\sqrt{N}/2}
\ \to \ 0
\end{equation}
as $N\to\infty$, we have 
$Z_1 / Z_2 \ \to \ 0$
as $N\to\infty$ with probability going to 1. By \eqref{eq:z1/z2}, this implies \eqref{eq:small_n}, so we are done.

%%%%%%%%%%%%%%%%%%%%%%%%%%%%%%%%%%%%%%%%%%%%%%%%%%%%%%%%%%%%%%%%%%%%%%%%%%%%%%%%%%%%%%%%%%%%%%%%%%%%%%%%%%%%%%%%%%%%%%
%%%%%%%%%%%%%%%%%%%%%%%%%%%%%%%%%%%%%%%%%%%%%%%%%%%%%%%%%%%%%%%%%%%%%%%%%%%%%%%%%%%%%%%%%%%%%%%%%%%%%%%%%%%%%%%%%%%%%%

\section{Continuous fragmentation with probabilistic stopping}
In this section, we consider the continuous breaking process in which the splitting number is fixed, but each new stick has a certain probability of becoming inactive. This is inspired by the conjecture on the discrete breaking process stopping at certain residue classes.

For simplicity, in the continuous breaking problem we always assume the initial length is 1, since scaling of stick lengths does not affect Benfordness. We can first consider the following simpler scenario. It can be seen as a generalization of the Restricted 1-Dimensional Decomposition Model studied in \cite[Theorem 1.9]{B+} (where they have shown the case $k = 2$). The proof is analogous.

\begin{theorem}
\label{thm:onestickfixratio}
    Start from a stick of length $L = 1$. Fix a positive integer $k\geq 2$. From a good distribution on $(0,1)$, sample $k-1$ values as cut points to break the stick into $k$ pieces. Then have exactly one stick be alive (i.e., all the remaining $k-1$ sticks become dead). Assume nothing about which stick is chosen to be alive. Repeat this $N$ times. Then as $N\to\infty$, the collection of resulting dead stick lengths converges to strong Benford behavior.
\end{theorem}
\begin{proof}
    Note that at the end of $N$ levels there will be $(k-1)N+1$ sticks in total. Since $\frac{(k-1)\log(N)}{N+1}\to 0$ as $N\to\infty$, we may remove the sticks that become dead in the first $\log(N)$ levels. The remaining sticks $X_i$ satisfy, uniformly,
    \begin{equation}
    \E[\varphi_s(X_i)] \ \to \ \log_{10}(s)
    \end{equation}
    as $N\to\infty$ by \Cref{thm:Mellin_transform_condition_continuity}. It follows that
    \begin{equation}
    \E[P_N(s)] \ \to \ \log_{10}(s)
    \end{equation}
    as $N\to\infty$. For the variance, we also adopt a similar strategy. Label the sticks that become dead at level $n$ as $X_{n(k-1)+i}$ for $1\leq i\leq k-1$. We say 
    a stick $X_i$ belongs to level $n$ if $n=\lceil \frac{i}{k-1}\rceil$, i.e., the stick becomes dead at level $n$.
    Consider the set of pairs of indices
    \begin{equation}
    \mathcal{A}\ := \ \{(i,j): (k-1)\log(N)+1\leq i\leq j - (k-1)(\log(N)+1)\leq N - (k-1)(\log(N)+1)\}.
    \end{equation}
    In other words, this is the collection of pairs that both become inactive after at least $\log(N)$ levels and are at least $\log(N)$ levels apart. The same reasoning as in the proof of \cite[Theorem 1.9]{B+} shows that for all $(i,j)\in\mathcal{A}$, we have
    \begin{equation} \label{one_stick_fix_ratio_product_calc}
        \E[\varphi_s(X_i)\varphi_s(X_j)] \ \to \ \log_{10}^2(s)
    \end{equation}
    as $N\to\infty$ uniformly. Explicitly, we have that
    \begin{eqnarray}
    X_i & \ = \ & \alpha_1 \cdots \alpha_{t-1} \alpha_t \cdots \alpha_u \nonumber\\
    X_j & \ = \ &\alpha_1 \cdots \alpha_{t-1} \beta_t \cdots \beta_v
    \end{eqnarray}
    where $v \ge t + \log(N)$ and $\alpha_{t+1}, \dots, \alpha_u, \beta_{t+1}, \dots, \beta_v$ are independent. Let $c = \alpha_1 \dots \alpha_{t-1} \beta_t$. We have that
    \begin{equation}
    \E[\varphi_s(X_i)\varphi_s(X_j) | \alpha_1, \dots, \alpha_u, \beta_t]\ =\  \varphi_s(X_i)\varphi_s(c \beta_{t+1} \dots \beta_v)
    \end{equation}
    which approaches $\varphi_s(X_i) \log_B(s)$ uniformly, with error a function of $N$. However, $\E[\varphi_s(X_i)] \to \log_B(s)$ uniformly as well, so \eqref{one_stick_fix_ratio_product_calc} follows. Moreover, it is easy to check that
    \begin{equation}
    |\mathcal{A}|=\frac{N^2}{2}+O(N\log (N)).
    \end{equation} Thus $\Var(P_N(s))\to 0$ as $N\to\infty$, as desired.
\end{proof}
A special case of the process in \Cref{conj:contbreakwstopratio} is the following.
\begin{theorem}
    Start from $R$ sticks, each of length $L=1$. Fix a positive integer $k\geq 2$. Initially all sticks are alive and each breaks into $k$ pieces independently, resulting in $kR$ new sticks. Then randomly choose $R$ out of these new sticks to continue, while the remaining $kR-R$ die. Repeat this for $N$ levels. Then as $N\to\infty$, the collection of resulting stick lengths converges to strong Benford behavior.
\end{theorem}
\begin{proof}
    This follows from essentially the same argument as in \Cref{thm:onestickfixratio}. For expectation, again notice that out of all $(k-1)RN+R$ final pieces, at most $(k-1)R\log(N)$ pieces have lengths being a product of less than $\log(N)$ independent ratios. For the remaining ones, we still have $\E[\varphi_s(X_i)]\to\log_{10}(s)$
    uniformly in $i$ as $N\to\infty$, and $\E[P_N(s)]\to\log_{10}(s)$ follows since $\frac{(k-1)R\log(N)}{(k-1)RN+R}\to 0$. For variance, note that given $X_i$ and $X_j$, $i<j$, belonging to levels $n_i$ and $n_j$ respectively, the only difference between the current scenario and the one in \Cref{thm:onestickfixratio} is that now they could come from different parents at  level $n_i$. Namely, the number of independent levels they have could now be larger than $n_j - n_i$. This makes the condition $\E[\varphi_s(X_i)\varphi_s(X_j)]\to \log_{10}^2(s)$ even easier to satisfy. Therefore, using the same argument of throwing away highly dependent pairs, we can easily generalize the proof to the current case.
\end{proof}

%%%%%%%%%%%%%%%%%%%%%%%%%%%%%%%%%%%%%%%%%%%%%%%%%%%%%%%%%%%
%%%%%%%%%%%%%%%%%%%%%%%%%%%%%%%%%%%%%%%%%%%%%%%%%%%%%%%%%%%
\subsection{Proof of Theorem \ref{conj:contbreakwstopratio}}
We first show that the process starting from a single stick terminates in finitely many levels.

\begin{proof}[Proof of finite termination]
    Let $p$ be the probability that it does not terminate. In such a case, one of the live children initiates a breaking that does not terminate. Thus, we have, if $A$ is the number of live children of the original stick,
\begin{align}
    p & \ = \ \sum_{a = 1}^k \prob(A = a)\prob(\text{at least one of the $a$ live children initiates infinite breaking}) \nonumber \\ \nonumber
    & \ = \ \sum_{a = 1}^k \binom{k}{a}\frac{1}{k^a}\left( 1-\frac{1}{k} \right)^{k-a} (1-(1-p)^a) \\ \nonumber
    & \ = \ \sum_{a = 1}^k \binom{k}{a}\frac{1}{k^a}\left( 1-\frac{1}{k} \right)^{k-a} - \sum_{a = 1}^k \binom{k}{a}\left( \frac{1-p}{k} \right)^a\left( 1-\frac{1}{k} \right)^{k-a} \\ \nonumber
    & \ = \ \left[ 1 - \left( 1-\frac{1}{k} \right)^k \right] - \left[ \left( 1-\frac{p}{k} \right)^k - \left( 1-\frac{1}{k} \right)^k \right] \\ 
    & \ = \ 1 - \left( 1-\frac{p}{k} \right)^k.
\end{align}
Now, we have that, by Bernoulli's inequality,
\begin{equation}
\left( 1-\frac{p}{k} \right)^k \ \ge \ 1 - p
\end{equation}
with equality if and only if $p = 0$. But we do have equality, so $p=0$, as desired.
\end{proof}

Now, consider the process where all $R$ sticks are being broken simultaneously. The above result implies that for any given $R$, this process also ends in finitely many levels with probability 1. Now we show the second part of \Cref{conj:contbreakwstopratio}.

Let $n_i$ be the number of live sticks present at the $i^{th} $ level so that $n_0 = R$. Then, we have the following:
\begin{lemma}
    For $i \ge 0$,
    \begin{equation}
    \prob(|n_i - R| \le t) \ \ge \ 1- \frac{2i^3R(k-1)}{t^2k}
    \end{equation}
    if $t < R$.
    \label{lem:tailBoundOnNi}
\end{lemma}
\begin{proof}
    The result is trivial for $i = 0$. We proceed with induction on $i$. Assume the result for $i$; we show it for $i+1$. Fix $n_i$. We have that 
    \begin{align}
        \prob(|n_{i+1} - R| \le t) & \nonumber \ \ge \ \prob\left(|n_i - R| \le \frac{i}{i+1}t, \ \ |n_{i+1} - n_i| \le \frac{1}{i+1}t\right) \\ \nonumber
        & \ \ge \ 1 - \prob\left( |n_i - R| > \frac{i}{i+1}t \right) - \prob\left( |n_{i+1} - n_i| > \frac{1}{i+1}t, \ \ |n_i - R| \le \frac{i}{i+1}t \right) \\ \nonumber
        & \ \ge \ 1 - \prob\left( |n_i - R| > \frac{i}{i+1}t \right) - \prob\left( |n_{i+1} - n_i| > \frac{1}{i+1}t, \ \ n_i < 2R \right) \\
        & \ \ge \ 1 - \prob\left( |n_i - R| > \frac{i}{i+1}t \right) - \prob\left( |n_{i+1} - n_i| > \frac{1}{i+1}t \ \Big| \  n_i < 2R \right). \label{inductive_tail_bound_alive_sticks}
    \end{align}
    Now, note that $n_{i+1}$ is binomially distributed with parameters $n_ik$ and $1/k$. Thus, conditioning on $n_i$, it has expectation $n_i$ and variance $n_i(1 - 1/k)$. So, by Chebyshev's inequality,
    \begin{equation}
    \prob\left( |n_{i+1} - n_i| > \frac{1}{i+1}t \ \Big| \  n_i \right) \ < \ \frac{n_i(1-1/k)}{\frac{1}{(i+1)^2}t^2} \ \le \ \frac{2(i+1)^2R(k-1)}{t^2k}
    \end{equation}
    and we have that, from \eqref{inductive_tail_bound_alive_sticks} and the inductive hypothesis,
    \begin{align}
        \prob(|n_{i+1} - R| \le t) & \ \ge \ 1 - \frac{2i^3R(k-1)}{\frac{i^2}{(i+1)^2}t^2k} - \frac{2(i+1)^2R(k-1)}{t^2k} \nonumber\\
        & \ \ge \ 1 - \frac{2(i+1)^3R(k-1)}{t^2k}.
    \end{align}
\end{proof}

For any $R$ and $N$, define
\begin{equation}P_{R}(s) \ :=\ \frac{\sum_i \varphi_s(X_i)}{\#\{X_i\}}\end{equation}
where the sum runs over the set of resulting sticks in a process starting with $R$ sticks (which is finite with probability 1).
We show $\E[P_R(s)]\to\log_{10}(s)$ and $\Var(P_R(s))\to 0$ as 
$R\to\infty$. 

For the expectation, we first show the existence of a function $h(R)\to\infty$ as $R\to\infty$ such that the average of $\E[\varphi_s(X_i)]$ for sticks $X_i$ that die within the first $h(R)$ levels goes to $\log_{B}(s)$ as $R\to\infty$. Define
\begin{equation}
P'_{R}(s) \ := \ \frac{\sum_{X_i \text{  in first $n$ levels}} \varphi_s(X_i)}{\#\{X_i | X_i \text{  in first $n$ levels}\}}.
\end{equation}
For $X_i$ belonging to level $n$, 
\begin{equation}
\left|\E[\varphi_s(X_i)] - \log_B(s) \right| \ \le \ f(n)
\end{equation}
where $f$ satisfies $f(n)\to 0$ as $n\to\infty$ by \Cref{thm:Mellin_transform_condition_continuity}. We now show that in each of the first $h(R)$ levels, a roughly equal number of sticks become dead. We may take $h(R)=R^{1/10}$ and $t=R^{2/3}$ and apply \Cref{lem:tailBoundOnNi}. Then we obtain that when $i\leq h(R)$,
\begin{equation}
\prob(R-R^{2/3}<n_i<R+R^{2/3}) \ \ge \ 1- \frac{2R^{3/10}R(k-1)}{R^{4/3} k} \ \ge \ 1 - 2R^{-1/30} \ \to \ 1
\end{equation}
as $R\to\infty$.
Let $\overline{n_i}$ be the number of sticks that become inactive at level $i$. Then we have for all $i\leq h(R)$, with probability going to 1,
\begin{equation}
(k-1)R-(k+1)R^{2/3} \ < \ \overline{n_i} \ < \ (k-1)R + (k+1)R^{2/3},
\end{equation}
which implies, when $R$ is sufficiently large,
\begin{equation}
    \label{eq:boundOnDeadSticks}
    \left(k-\frac{3}{2}\right)R \ < \ \overline{n_i} \ < \ \left(k-\frac{1}{2}\right)R.
\end{equation}
Thus, conditioning on the above event,
\begin{align}
    |\E(P'_R(s)) - \log_B(s)| 
&\ \le \
\frac{1}{\sum_{i=1}^{h(R)}\overline{n_i}}\sum_{i=1}^{h(R)}f(i)\overline{n_i} \nonumber \\
&\ \le \
\frac{1}{(k-\frac{3}{2})Rh(R)}
\sum_{i=1}^{h(R)}f(i)(k-\frac{1}{2})R \nonumber \\
&\ \le \
3 \frac{1}{h(R)}\sum_{i=1}^{h(R)}f(i) \ \to \ 0
\end{align}
as $R\to\infty$.
This implies
$\E[P'_R(s)]\to\log_{10}(s)$. Now, for the sticks after level $h(R)$, simply note that
\begin{equation}
\left|\E[\varphi_s(X_i)] - \log_B(s) \right| \ \le \ f(n) \ \le \ \inf_{n \ge h(R)} f(n)
\end{equation}
which tends to $0$ as $R \to \infty$. $P_R(s)$ is a weighted average of these $\varphi_s(X_i)$ and $P'_R(s)$, so $|\E(P_R(s)) - \log_B(s)| \to 0$, as desired.

Now we analyze the variance. Since for any pair of final sticks $(X_i,X_j)$, if both die after at least $\log(R)$ levels and they die at least $\log(\log(R))$ levels apart (i.e., they have enough independence), then we have
\begin{equation}
\E[\varphi_s(X_i)\varphi_s(X_j)] \ \to \ \log_{B}^2(s)
\end{equation}
uniformly for all pairs satisfying these criteria by \Cref{thm:Mellin_transform_condition_continuity}. Thus, it suffices to show that the proportion of pairs of the following two types among all pairs goes to 0 as $R\to\infty$:
\begin{enumerate}
    \item at least one of $X_i$,$X_j$ dies before $\log(R)$ levels, or
    \item $X_i$, $X_j$ have a common ancestor less than $\log(\log(R))$ levels before they both die.
\end{enumerate}

Let $M$ be the total number of dead sticks ever.

To show (1), we first show that the number of sticks that die within the first $\log(R)$ levels is small compared to $M$ with probability going to 1. Keep our choice of $h(R)$ and $t$ earlier. 
Therefore when $R$ is sufficiently large, using the upper bound from \eqref{eq:boundOnDeadSticks}, we get that as number of sticks that die within the first $\log(R)$ levels is upper bounded by
\begin{equation}
\left(k-\frac{1}{2}\right)\log(R)R
\end{equation}
with probability going to 1.
Now again using \eqref{eq:boundOnDeadSticks}, we can lower bound $M$ by lower bounding the total number of sticks that die within the first $h(R)$ levels. This gives
\begin{equation}
\label{eq:lowerBoundOnM}
    M \ \geq \ h(R)\left(k-\frac{3}{2}\right)R \ = \ \left(k-\frac{3}{2}\right)R^{11/10}
\end{equation}
with probability going to 1. 
Since
\begin{equation}
\frac{(k-\frac{1}{2})\log(R)R}{(k-\frac{3}{2})R^{11/10}} \ \to \ 0
\end{equation}
as $R\to\infty$, we have shown that the proportion of sticks that die in the first $\log(R)$ levels among all goes to 0 as $R\to\infty$ with probability going to 1. This then implies that the number of pairs that involve a stick of this type also takes up a diminishing proportion of all pairs of final sticks as $R\to\infty$.

Now we show (2), namely, that the number of pairs $X_i$, $X_j$ having a common ancestor at most $\log(\log R)$ levels before they both die is $o(M^2)$ with high probability. Fix some $X_i$. Then, the number of sticks, dead or alive, that share the $\alpha$ ancestor of $X_i$ and is $\alpha - \beta$ levels away is at most $k^\beta$. Thus, the number of $X_j$ that satisfying (2) when paired with $X_i$ is bounded above by
\begin{equation}
    \sum_{\alpha=1}^{\lfloor \log(\log R) \rfloor} \sum_{\beta = 0}^{\lfloor \log(\log R) \rfloor} k^\beta \ \le \ \log(\log R) \frac{k^{\log(\log R)}-1}{k-1} \ \le \ \log(\log R) (\log R)^{\log k}.
\end{equation}
Hence, the number of such pairs is bounded above by $M(\log R)^{1 + \log k}=o(M^2)$ by \eqref{eq:lowerBoundOnM}.

%\xinyu{The other idea that Miller suggested was to consider extreme cases (when the proportions of sticks that die each level are all small/big), and then try to use these cases to bound most normal cases.}

%%%%%%%%%%%%%%%%%%%%%%%%%%%%%%%%%%%%%%%%%%%%%%%%%%%%%%%%%%%
%%%%%%%%%%%%%%%%%%%%%%%%%%%%%%%%%%%%%%%%%%%%%%%%%%%%%%%%%%%
\subsection{Proof of Theorem \ref{thm:contR>1/k}}

Let $A$ be some integer that is sufficiently large (we can determine what this means later). There then exists some fixed $j$ such that $n_j > A$ with positive probability $p^*$. Now, consider $i \ge j$. Conditioning on $n_i$, we have that $n_{i+1}$ is a random variable with mean $n_irk$ and variance $n_ir(1-r)$. Thus, by Chebyshev's inequality, we have that
\begin{equation}
    \prob\left(n_{i+1} > n_i\left( 1 + \frac{rk-1}{2} \right) \right) \ \ge \ 1 - \prob\left( |n_{i+1}-n_irk| \ge n_i \frac{rk-1}{2} \right) \ \ge \ 1 - \frac{n_ir(1-r)}{n_i^2 \left(\frac{rk-1}{2}\right)^2}.
\end{equation}
We can then let $a = \frac{r(1-r)}{A \left(\frac{rk-1}{2}\right)^2}$ and $c = 1 + \frac{rk-1}{2}$. Then the above inequality can be written as 
\begin{equation}
    \prob(n_{i+1} > cn_i) \ \ge \ 1 - \frac{aA}{n_i}.
\end{equation}
It follows that
\begin{align}
    \prob(n_{i+1} > Ac^{i-j+1} \ \big| \ n_i > Ac^{i-j})
& \nonumber \ \ge \
\prob(n_{i+1}>c n_i\ \big| \ n_i > Ac^{i-j}) \\ &\ \ge \
\inf_{n_i>Ac^{i-j}}\left(
1 - \frac{aA}{n_i}\right)  
\ \ge \ 1 - ac^{j-i}.
\end{align}
Hence, the probability that $n_i > Ac^{i-j}$ for all $i \ge j$ given that $n_j > A$ is at least
\begin{equation}
p' \ = \ (1-a)(1-ac^{-1})(1-ac^{-2}) \cdots.
\end{equation}
Now, since $\lim_{x \to 0} \log(1-x)/x = -1$, we may set $A$ large enough so that $a$ is sufficiently small so that $\log(1-ac^{t}) > -2ac^t$ for $t \le 0$. We then have
\begin{equation}
\log(p') \ = \ \sum_{t = 0}^\infty \log(1-ac^{-t}) \ > \ \sum_{t = 0}^\infty -2ac^{-t} \ = \ -\frac{2a}{1-c}.
\end{equation}
In particular, $p' \ge e^{-2a/(1-c)} > 0$. 
Thus the probability that $n_i > Ac^{i-j}$ for all $i \ge j$ is at least $p^*p'$ which is positive. 
Hence, not only is the process infinite with positive probability, but the number of alive sticks at each level blows up with positive probability.

%%%%%%%%%%%%%%%%%%%%%%%%%%%%%%%%%%%%%%%%%%%%%%%%%%%%%%%%%%%
%%%%%%%%%%%%%%%%%%%%%%%%%%%%%%%%%%%%%%%%%%%%%%%%%%%%%%%%%%%
\subsection{Proof of Theorem \ref{conj: contbreaks_dependent}}
Without assuming independence on the alive/dead status of the sticks, we prove the following weaker version of \Cref{lem:tailBoundOnNi}.
\begin{lemma}
\label{lem:dependentTailBound}
    For $i \ge 0$,
    \begin{equation}
    \prob(|n_i - R| \le t) \ \ge \ 1- \frac{2i^3Rk^2}{t^2}
    \end{equation}
    if $t < R$.
\end{lemma}

\begin{proof}
    As in the proof of \Cref{lem:tailBoundOnNi}, we proceed with induction on $i$, noting that the result is trivial for $i = 0$. By the same calculation, \eqref{inductive_tail_bound_alive_sticks} holds, That is,
    \begin{equation}
    \prob(|n_i - R| \le t) \ \ge \ 1 - \prob\left( |n_i - R| > \frac{i}{i+1}t \right) - \prob\left( |n_{i+1} - n_i| > \frac{1}{i+1}t \ \Big| \  n_i < 2R \right).
    \end{equation}
    We have that $n_{i+1}$ is the sum of $n_i$ independent random variables with mean $1$ and variance bounded by $k^2$. Thus, conditioning on $n_i$, it has expectation $n_i$ and variance at most $n_ik^2$. Chebyshev's inequality implies
    \begin{equation}
    \prob\left( |n_{i+1} - n_i| > \frac{1}{i+1}t \ \ \Big| \  \ n_i \right) \ < \ \frac{n_ik^2}{\frac{1}{(i+1)^2}t^2} \ \le \ \frac{2(i+1)^2Rk^2}{t^2}.
    \end{equation}
    We then have that
    \begin{equation}
    \prob(|n_{i+1} - R| \le t) \ \ge \ 1 - \frac{2i^3Rk^2}{\frac{i^2}{(i+1)^2}t^2k} - \frac{2(i+1)^2Rk^2}{t^2k} \ \ge \ 1- \frac{2(i+1)^3Rk^2}{t^2}
    \end{equation}
    which completes the induction.
\end{proof}
\Cref{conj: contbreaks_dependent}
 follows from essentially the same arguments as in proof of \Cref{conj:contbreakwstopratio} using \Cref{lem:dependentTailBound}. We highlight the necessary changes below.
 
For any $R$ and $N$, define
\begin{equation}P_{R,N}(s) \ := \ \frac{\sum_i \varphi_s(X_i)}{\#\{X_i\}}\end{equation}
where the sum runs over the set of resulting sticks in the first $N$ levels of a process starting with $R$ sticks. We prove $\E[P_{R,N}(s)]\to\log_{10}(s)$ and $\Var(P_{R,N}(s))\to 0$ if $N\geq \log(R)$ and $R\to\infty$. Keep the choices of $h(R) = R^{1/10}$ and $t=R^{2/3}$ in the proof of \Cref{conj:contbreakwstopratio}. For sticks that die after $h(R)$ levels, we know that 
\begin{equation}
\left|\E[\varphi_s(X_i)] - \log_B(s) \right| \ \le \ f(h(R))
\end{equation}
where the right-hand-side goes to 0 in $R$. Therefore it again suffices to estimate the errors
\begin{equation}
\left|\E[\varphi_s(X_i)] - \log_B(s) \right|
\end{equation}
for $X_i$ that dies within the first $h(R)$ levels. Now the exact same argument applies simply after replacing \Cref{lem:tailBoundOnNi} with \Cref{lem:dependentTailBound}.

For variance, let $M$ now denote the number of resulting sticks after $N$ levels. By the same logic,
\begin{equation}
\E[\varphi_s(X_i)\varphi_s(X_j)] \ \to \ \log_B^2(s)
\end{equation}
uniformly given that $X_i$ and $X_j$ have a most recent common ancestor more than $\log(\log R)$ levels away from $X_i$ and both $X_i, X_j$ die after at least $\log(R)$ levels. It therefore suffices to show that such pairs $X_i, X_j$ make up a proportion of all pairs of dead sticks that tends to $1$. This is done in the same way as in the proof of \Cref{conj:contbreakwstopratio}.

%%%%%%%%%%%%%%%%%%%%%%%%%%%%%%%%%%%%%%%%%%%%%%%%%%%%%%%%%%%
%%%%%%%%%%%%%%%%%%%%%%%%%%%%%%%%%%%%%%%%%%%%%%%%%%%%%%%%%%%
\subsection{Proof of Theorem \ref{thm: contbreaks_dependent_r_small}}

We now argue that there is a limiting distribution for the stick lengths exists. This implies that the mantissas also approach some distribution.

\begin{lemma}
\label{lem:limitExists}
    When $r\leq 1/k$, the final collection of stick lengths converges to a unique limiting distribution as $R\to\infty$.
\end{lemma}
\begin{proof}
    Fix $R$. Note that the distribution of the overall collection of sticks resulting from breaking $R$ initial sticks is an average of the distributions of sticks resulting from each initial stick weighted by the number of resulting sticks. For a single-stick breaking process, let $p_i$ be the probability that it ends with $i$ sticks. Assume that given that the process ends in $i$ sticks, the collection of stick lengths follow distribution $\mathcal{D}_i$, with cumulative distribution function $F_{\mathcal{D}_i}$ that is continuous and compactly supported. Namely, a random variable representing the length of a randomly chosen stick among the $i$ ending sticks follows $\mathcal{D}_i$. Then the weighted average has cumulative distribution function
    \[
    F\ =\ \frac{\sum_{i=1}^\infty i p_i F_{\mathcal{D}_i}}{\sum_{i=1}^\infty i p_i}.
    \]
    This sum converges because $\sum_{i=1}^\infty i p_i = \E[M_R]$, which is finite by \Cref{lem: r_small_stick_expectation}. Now it suffices to show our claim about $\mathcal{D}_i$. Fix $i$. There are a finite number of ways in which sticks die that result in a collection of $i$ sticks when the process ends. For each given configuration (i.e., sequence of dying of the sticks), the probability density function of the length of a particular final stick is a certain integral of the product of the density functions of the corresponding component for each ancestor of the break-point distribution. Averaging this over all configurations resulting in $i$ sticks, due to symmetry, every stick ends up having the same distribution. Therefore $\mathcal{D}_i$ is well-defined. Then take the limit as $R\to\infty$.
    Use \emph{Borel's law of large numbers}\footnote{See for example \textcolor{blue}{\url{https://en.wikipedia.org/wiki/Law\_of\_large\_numbers\#Borel's\_law\_of\_large\_numbers}  }.} to show that for a large enough collection of $i$'s, when $R$ is large enough, with probability 1, the proportion of trials ending in $i$ sticks is close to $p_i$; moreover, when $R$ is large enough, the cdf of the lengths within each group of trials resulting in $i$ sticks  is close to $F_{\mathcal{D}_i}$ with probability 1. Just need to show point-wise convergence!
\end{proof}

\begin{lemma} \label{lem: r_small_stick_expectation}
    We have that
    \begin{equation}
        \E[M_R] \ = \ R\frac{k-kr}{1-kr}
    \end{equation}
\end{lemma}
\begin{proof}
    Let $p_i$ be the probability that exactly $i$ of the children of the first stick are alive. Then,
    \begin{equation}
    \begin{split}
        \E[M_1] \ &= \ \sum_{i=0}^k p_i (\E[M_i] + k-i) \ = \ \sum_{i=0}^k p_i(i\E[M_1] + k-i) \\
        &= \ k + (\E[M_1] - 1) \sum_{i=0}^k ip_i \ = \ k - kr + kr\E[M_1]
    \end{split}
    \end{equation}
    so that
    \begin{equation}
        \E[M_1] \ = \ \frac{k-kr}{1-kr}.
    \end{equation}
    Linearly of expectation then implies the result.
\end{proof}

\begin{corollary} \label{cor: r_small_sticks_small}
    We have that
    \begin{equation}
        M_R \ \le \ 2R\frac{k-kr}{1-kr}
    \end{equation}
    with probability at least $1/2$.
\end{corollary}
\begin{proof}
    This follows directly from \Cref{lem: r_small_stick_expectation} and Markov's inequality.
\end{proof}

\begin{lemma} \label{lem: r_small_first_level}
    Let $a > 1$ be some real number and let $b_a$ be the expected number of child sticks that are of length at least $L/a$ starting from a stick of length $L$. With probability at least 
    \begin{equation}
        1 - \frac{4k^2}{b_a^2(1-r)^2R}
    \end{equation}
    the number of dead sticks in the first level of length at least $L/a$ is at least $b_a(1-r)R/2$.
\end{lemma}
\begin{proof}
    Denote this quantity by $M^R_{L,a}(1)$. Then, the probability of a child with length at least $L/a$ being dead is $1-r$. Thus, 
    \begin{equation}
        \E[M^R_{L,a}(1)] \ = \ Rb_a(1-r).
    \end{equation}
    Note that $M^R_{L,a}(1)$ is a sum of independent random variables distributed identically to $M^1_{L,a}(1)$, and $\Var[M^1_{L,a}(1)] \le k^2$, so
    \begin{equation}
        \Var[M^R_{L,a}(1)] \ \le \ Rk^2.
    \end{equation}
    Chebyshev's inequality then implies
    \begin{equation}
        \prob\left(M^R_{L,a}(1) \le \frac{b_a}{2}(1-r)R \right) \ \le \ \frac{Rk^2}{(Rb_a(1-r)/2)^2} \ = \ \frac{4k^2}{b_a^2(1-r)^2R}.
    \end{equation}
\end{proof}

By \Cref{lem: r_small_first_level} and \Cref{cor: r_small_sticks_small}, the proportion of sticks with length at least $L/a$ is at least
\begin{equation}
    \frac{b_a}{2}(1-r)R\left( 2R\frac{k-kr}{1-kr} \right)^{-1} \ = \ \frac{b_a}{4}(1-r) \frac{1-kr}{k-kr} \ = \ \frac{b_a(1-kr)}{4k}
\end{equation}
with probability at least 
\begin{equation}
    \frac{1}{2} - \frac{4k^2}{b_a^2(1-r)^2R}.
\end{equation}
Now, choose $a = k$. Then, we have that, with some probability approaching $1/2$, at least a proportion of $b_k(1-kr)/(4k)$ of the sticks are of length at least $L/k$. Moreover, $b_k > 0$ so this proportion is positive. Let
\begin{equation}
    B \ > \ k^{\frac{4k}{b_k(1-kr)}}.
\end{equation}
Then, these sticks occupy an interval of length
\begin{equation}
    \log_B(k) \ < \ \frac{b_a(1-kr)}{4k}
\end{equation}
in the distribution of the normalized mantissas. It follows that the mantissas of the stick lengths do not almost surely approach a uniform distribution as $R \to \infty$. That is, the stick lengths do not approach Benford behavior.

%%%%%%%%%%%%%%%%%%%%%%%%%%%%%%%%%%%%%%%%%%%%%%%%%%%%%%%%%%%%%%%%%%%%%%%%%%%%%%%%%%%%%%%%%%%%%%%%%%%%%%%%%%%%%%%%%%%%%%
%%%%%%%%%%%%%%%%%%%%%%%%%%%%%%%%%%%%%%%%%%%%%%%%%%%%%%%%%%%%%%%%%%%%%%%%%%%%%%%%%%%%%%%%%%%%%%%%%%%%%%%%%%%%%%%%%%%%%%

\section{Discrete fragmentation with congruence stopping condition}
In this section, we consider the setting of discrete stick fragmentation, i.e., all stick lengths involved are positive integers. We present the proofs of \Cref{thm:evenStopBreaking}, where we start with a stick of odd integer length, break it in two each time and only continue breaking the new stick with odd length until we reach a stick of length 1, and \Cref{thm:discreteStopAtHalfResidues}, where we start with a collection of $R$ sticks and break each of them in two until a piece falls into certain residue classes or becomes length 1. When we take the appropriate limit in both of these scenarios, the ending collection of stick lengths becomes Benford. The overall strategy, adopted from that in \S 3 of \cite{B+}, is to approximate the discrete process with an appropriate continuous analogue, and by showing that the two processes are ``close'' in a precise sense, deduce the desired result from the corresponding continuous result.

%%%%%%%%%%%%%%%%%%%%%%%%%%%%%%%%%%%%%%%%%%%%%%%%%%%%%%%%%%%
%%%%%%%%%%%%%%%%%%%%%%%%%%%%%%%%%%%%%%%%%%%%%%%%%%%%%%%%%%%
\subsection{Proof of Theorem \ref{thm:evenStopBreaking}}
In order to carry out the approximation strategy outlined above, we define a continuous process and a discrete process based on the same sequence of random ratios, so that the latter is the process we are interested in and the former known to be Benford. Our goal is to show that that their end results are ``close'' enough so that the Benfordness of the former implies that of the latter. The two processes are defined as follows.
Let $(c_i)_{i\geq 0}$ be a sequence of random numbers chosen from $(0,1)$ with respect to the uniform distribution. 
\begin{itemize}
    \item Let $\mathcal{Q}$ denote the continuous process. In this process, we start with a stick of length $h_0 = L$. For each $i\geq 1$, break off a fragment of length $Y_i = c_{i-1}h_{i-1}$ at the $i$-th level, which becomes \emph{dead}, namely, stops breaking further. The other stick of length $h_i = h_{i-1} - Y_i=(1-c_{i-1})h_{i-1}$ stays \emph{alive} and continues to break in the next step.
    \item Let $\mathcal{P}$ denote the discrete process. In this process, we start with a stick of length $\ell_0 = L$. For each $i\geq 1$, break off a fragment of length $X_i = 2\lceil\frac{c_{i-1}(\ell_{i-1}-1)}{2}\rceil$ at the $i$-th level, which becomes \emph{dead}. Note that by construction, $X_i$ is an even integer taking values in $[2, \ell_{i-1}-1]$.The remaining stick of length $\ell_i = \ell_{i-1} - X_i$ (which is always an odd integer) stays \emph{alive} and continues to break in the next step. 
    \item Moreover, a stick in $\mathcal{P}$ also becomes dead if it has length 1. In that case, the corresponding stick in $\mathcal{Q}$ also dies.
\end{itemize}
We first derive the following lemma that bounds the length of a stick $X_k$ in $\mathcal{P}$ with the length of the corresponding stick $Y_k$ in $\mathcal{Q}$.
\begin{lemma}
\label{lem:XandYclose}
Given that $\ell_k, h_k > 2$, we have,
\begin{equation}
Y_{k} \prod_{i=1}^{k-1} \left( 1 - \frac{2}{\ell_{i}-2} \right) - 2 \ \le \ X_k \ \le \ Y_{k} \prod_{i=1}^{k-1} \left( 1 + \frac{2}{\ell_{i}-2} \right) + 2\prod_{i=1}^{k-1} \frac{\ell_i}{\ell_i-4}.
\end{equation}
\end{lemma}
\begin{proof}
Let $d_k$ be the rounded version of $c_k$ used in $\mathcal{P}$, i.e., $d_k = X_{k+1}/\ell_k$. Then, note that
\begin{equation}
(\ell_k-1)c_k \ \le \ X_{k+1} \ \le \ (\ell_k-1)c_k + 2
\end{equation}
so that
\begin{equation}
\left( 1 - \frac{1}{\ell_k} \right) c_k \ \le \ d_k \ \le \ \left( 1 - \frac{1}{\ell_k} \right) c_k + \frac{2}{\ell_k} \quad \implies \quad |d_k - c_k| \ \le \ \frac{2}{\ell_k}.
\end{equation}
It follows that
\begin{equation}
    \ell_k \ = \ L\prod_{i=0}^{k-1} (1-d_i) \ \le \ L \prod_{i=0}^{k-1} \left( 1 - c_i + \frac{2}{\ell_i} \right) \ \le \ L\prod_{i=0}^{k-1} (1-c_i) \prod_{i=0}^{k-1} \left( 1 + \frac{2}{\ell_i(1-c_i)} \right).
\end{equation}
We have that 
\begin{equation}
    \ell_i(1-c_i) \ \ge \ \ell_i(1-d_i) - 2 \ = \ \ell_{i+1} - 2,
    \label{eq:helper}
\end{equation}
 so
\begin{equation} \label{upper_bound_ell_k}
    \ell_k \ \le \ L\prod_{i=0}^{k-1} (1-c_i) \prod_{i=0}^{k-1} \left( 1 + \frac{2}{\ell_{i+1}-2} \right) \ \le \ h_k\prod_{i=1}^{k} \left( 1 + \frac{2}{\ell_{i}-2} \right).
\end{equation}

Equation \eqref{eq:helper} also implies that $1-c_i-\frac{2}{\ell_i}\geq 1-c_i\left(1-\frac{2}{\ell_{i+1}-2}\right)$, so that 
\begin{equation} \label{lower_bound_ell_k}
    \ell_k \ \ge \ L \prod_{i=0}^{k-1} \left( 1 - c_i - \frac{2}{\ell_i} \right) 
    \ \ge \ L\prod_{i=0}^{k-1}\left[ (1-c_i)\left( 1 - \frac{2}{\ell_{i+1}-2} \right)\right]
    \ \ge \ h_k\prod_{i=1}^{k} \left( 1 - \frac{2}{\ell_{i}-2} \right).
\end{equation}

We can then multiply \eqref{upper_bound_ell_k} by $d_k$ to get
\begin{equation}
    X_{k+1} \ \le \ h_kd_k \prod_{i=1}^{k-1} \left( 1 + \frac{2}{\ell_{i+1}-2} \right) \ \le \ \left(Y_{k+1} + \frac{2h_k}{\ell_k} \right)\prod_{i=1}^{k-1} \left( 1 + \frac{2}{\ell_{i+1}-2} \right)
\end{equation}
and then use \eqref{lower_bound_ell_k} to obtain
\begin{align}
    X_{k+1} \nonumber \ &\le \ Y_{k+1} \prod_{i=1}^{k} \left( 1 + \frac{2}{\ell_{i}-2} \right) + 2\prod_{i=1}^{k} \left( 1 + \frac{2}{\ell_{i}-2} \right) \left( 1 - \frac{2}{\ell_{i}-2} \right)^{-1} \\
    \ &\le \ Y_{k+1} \prod_{i=1}^{k} \left( 1 + \frac{2}{\ell_{i}-2} \right) + 2\prod_{i=1}^{k} \frac{\ell_i}{\ell_i-4}.
\end{align}
We can reason similarly by multiplying \eqref{lower_bound_ell_k} with $d_k$ to obtain
\begin{equation}
    X_{k+1} \ \ge \ Y_{k+1} \prod_{i=1}^{k} \left( 1 - \frac{2}{\ell_{i}-2} \right) - 2.
\end{equation}
\end{proof}

Let $g(x)$ be a function that goes to infinity as $x\to\infty$ with $g(x)=o(\sqrt{\log(x)})$. Let $h(x)$ be a function that goes to infinity as $x\to\infty$. The following corollary of \Cref{lem:XandYclose} essentially says that $X_k$ and $Y_k$ are very close given that $k$ is not too large and $\ell_{k-1}$, $Y_k$ are large enough.
\begin{corollary}
\label{cor:XandYclose}
    For all $k<g(L)\log L$ such that $\ell_{k-1} > \log^2(L)+2$ and $Y_k > h(L)$, we have
    \begin{equation}
    Y_k(1-o(1)) \ \leq \ X_k \ \leq \ Y_k(1+o(1)).
    \end{equation}
    \begin{proof}
        By Lemma \ref{lem:XandYclose}, we have
        \begin{align}
            X_k 
            & \nonumber \ \le \ Y_{k} \prod_{i=1}^{k-1} \left( 1 + \frac{2}{\ell_{i}-2} \right) + 2\prod_{i=1}^{k-1} \frac{\ell_i}{\ell_i-4}\\ \nonumber 
            & \ \le \ Y_k\left(1+\frac{2}{\log^2(L)}\right)^{k-1}+2\left(\frac{\log^2(L)}{\log^2(L)-4}\right)^{k-1}\\ \nonumber 
            & \ \le \ Y_k\left(1+\frac{2}{\log^2(L)}\right)^{g(L)\log L}+2\left(1 + \frac{8}{\log^2(L)}\right)^{g(L)\log L}\\
            & \ \le \ Y_k \exp\left({\frac{2g(L)}{\log L}}\right) + 2\exp\left(\frac{8g(L)}{\log L}\right).
        \end{align}
        As $L\to\infty$, $\frac{g(L)}{\exp(L)}\to 0$, so $\exp\left({\frac{2g(L)}{\log L}}\right)\to 1$ and $2\exp\left(\frac{8g(L)}{\log L}\right)=O(1)$. Now by our assumption $Y_k\to\infty$, we get asymptotically that 
        \begin{equation}
            X_k \ \leq \ Y_k(1+o(1)).
        \end{equation}
        For the other inequality, apply Lemma \ref{lem:XandYclose} again to get
        \begin{align}
            X_k 
            & \nonumber  \ \ge \ Y_{k} \prod_{i=1}^{k-1} \left( 1 - \frac{2}{\ell_{i}-2} \right) - 2\\ \nonumber 
            & \ \ge \ Y_k\left(1+\frac{2}{\log^2(L)}\right)^{k-1}-2\\ \nonumber 
            & \ \ge \ Y_k\left(1-\frac{2}{\log^2(L)}\right)^{g(L)\log L}-2\\
            & \ \ge \ Y_k (2e)^{-{\frac{2g(L)}{\log L}}}-2.
        \end{align}
        Again $(2e)^{-{\frac{2g(L)}{\log L}}}\to 1$ as $L\to\infty$, and since $Y_k\to\infty$,
        we get asymptotically that 
        \begin{equation}
            X_k \ \geq \ Y_k(1-o(1)).
        \end{equation}.
    \end{proof}
\end{corollary}
The following lemma then helps us translate Benfordness of $\{Y_i\}$ to that of $\{X_i\}$ given that they are close enough in the sense above. This is essentially \cite[Lemma 3.3]{B+}, but we give a different proof here.
Let $\{Z_i\}_L = \{Z_1,\dots,Z_{k_L}\}$ denote a finite sequence of random variables whose length $k_L$ depends on $L$. 
\begin{lemma}
\label{lem:smallPerturbPresBenford}
Suppose $\{Y_i\}_L = \{Y_1, Y_2, \dots, Y_{k_L}\}$ is strong Benford as $L \to \infty$. Then if $\{X_i\}_L = \{X_1, X_2, \dots, X_{k_L}\}$ is such that 
\begin{equation}Y_i(1 - o(1))\ \leq\ X_i\ \leq\ Y_i(1 + o(1))\end{equation} 
as $L\to\infty$, $\{X_i\}_L$ is strong Benford as $L \to \infty$.
\end{lemma}
\begin{proof}
    We prove that $\log(X_i)\mod 1$ is equidistributed in $[0,1]$. For simplicity, define \begin{equation}\phi(x) \ := \ \log(x)\pmod 1\end{equation} for any $x>0$.
    By our assumption we have
    \begin{align}
    & \nonumber \log(Y_i) + \log(1-o(1)) \ \leq \ \log(X_i) \ \leq \ \log(Y_i) + \log(1+o(1)) \\ \nonumber
    & \implies \log(Y_i)-o(1) \ \leq \ \log(X_i) \ \leq \ \log(Y_i)+o(1) \\
    & \implies \phi(Y_i)-o(1) \ \leq \ \phi(X_i) \ \leq \ \phi(Y_i)+o(1)
    \end{align}
    % below is the formatting from before
    % \begin{equation}
    % \log(Y_i) + \log(1-o(1)) \ \leq \ \log(X_i) \ \leq \ \log(Y_i) + \log(1+o(1))
    % \end{equation}
    % \begin{equation}
    % \implies \log(Y_i)-o(1) \ \leq \ \log(X_i) \ \leq \ \log(Y_i)+o(1)
    % \end{equation}
    % \begin{equation}
    % \implies \phi(Y_i)-o(1) \ \leq \ \phi(X_i) \ \leq \ \phi(Y_i)+o(1)
    % \end{equation}
    with probability going to 1.
    For any $0\leq a<b\leq 1$,
    \begin{equation}
    \prob\left(a+o(1)<\phi(Y_i)<b-o(1)\right)
    \ \leq \
    \prob(a<\phi(X_i)<b)
    \ \leq \
    \prob(a-o(1)<\phi(Y_i)<b+o(1))
    \end{equation}
    with probability going to 1. But since $Y_i$ is strong Benford, we have that 
    \begin{equation}
    \prob\left(a+o(1) < \phi(Y_i) < b-o(1)\right) \ = \ b-a-o(1)
    \end{equation}
    and 
    \begin{equation}
    \prob\left(a-o(1) <\phi(Y_i)<b+o(1)\right) \ = \ b-a+o(1),
    \end{equation}
    so 
    \begin{equation}
    b-a-o(1) \ \leq \ \prob(a<\phi(X_i)<b) \ \leq \ b-a+o(1),
    \end{equation}
    which implies that $\prob(a<\phi(X_i)<b)\to b-a$ as $L\to\infty$ with probability going to 1.
\end{proof}
% \begin{proof}
%     We fix a digit $j$ and show that the distribution of the $j$\textsuperscript{th} digit (counting from left to right) is Benford. Let $D_j(a)$ denote the $j$\textsuperscript{th} digit of $a$. Then since $X_{i}=Y_i(1\pm o(1))$, there exists some function of $L$ tending to infinity with $L$ such $D_j(X_i)= D_j(Y_i)$, unless $D_{j+k}(Y_i)=9$ for all $1\leq k\leq f(L)$ or $D_{j+k}(Y_i)=0$ for all $1\leq k\leq f(L)$. Since the $Y_i$ are strong Benford and are therefore Benford in each $(j+k)$\textsuperscript{th} digit, this occurs with probability going to 0.
% \end{proof}
By \cite[Theorem 1.9]{B+}, the process $\mathcal{Q}$ is Benford. 
Given the lemma above, it now suffices to show that the premises of \Cref{cor:XandYclose} are satisfied for almost all $k$. The following lemma shows that the process ends within $g(L)\log L$ levels with probability going to 1, so the first condition that $k$ is not too large is almost always true.
\begin{lemma}
\label{lem:countfragments}
Let $F_L$ be the number of fragments generated by a stick of length $L$. As $L \to \infty$,
\begin{equation}
    \prob[(\log\log L)^2 < F_L < g(L)\log L]\ =\ 1-o(1).
\end{equation}
\end{lemma}

\begin{proof}
We first show the upper bound using Markov's inequality. We prove by induction that
\begin{equation}
\label{eq:expectation}
    \E[F_\ell]\ =\ 1 + 2\sum_{\substack{0< j < \ell \\ j \text{ even}}}\frac{1}{j}.
\end{equation}
It is clear that $\E[F_1] = 1$. 
We have the recurrence
\begin{equation}
    \E[F_L]\ =\ \frac{2}{L-1} \sum_{\substack{\ell < L \\ \ell \text{ odd}}} (1 + \E[F_\ell])
\end{equation}
since there is a $\frac{2}{L-1}$ probability of breaking off a piece of length $\ell$ in the first break for $1 \le \ell \le L-1$ and $\ell$ odd. By the induction hypothesis, we have
\begin{align}
    \E[F_L] 
    & \nonumber \ =\ \frac{2}{L-1} \sum_{\substack{\ell < L \\ \ell \text{ odd}}} \left(1 + \left(1 + 2\sum_{\substack{0< j < \ell \\ j \text{ even}}}\frac{1}{j}\right)\right)\\ \nonumber
    &\ =\ \frac{2}{L-1} \cdot \frac{L-1}{2} + \frac{2}{L-1} \sum_{\substack{\ell < L \\ \ell \text{ odd}}} \left(1 + 2\sum_{\substack{0< j < \ell \\ j \text{ even}}}\frac{1}{j}\right) \\  \label{ln:rearrange} \nonumber
    &\ =\ 1 + \frac{2}{L-1}\left( \frac{L-1}{2}+ 2\sum_{\substack{0< j < L-2 \\ j \text{ even}}} \frac{\frac{L-j-1}{2}}{j}\right)\\ \nonumber
    &\ =\ 1 + \frac{2}{L-1}\left( 1 + \sum_{\substack{0< j < L-2 \\ j \text{ even}}} \left(1+ \frac{L-j-1}{j}\right)\right)\\ \nonumber
    &\ =\ 1 + \frac{2}{L-1} + \frac{2}{L-1}\sum_{\substack{0< j < L-2 \\ j \text{ even}}} \frac{L-1}{j}\\
    &\ =\ 1 + 2 \sum_{\substack{0< j < L \\ j \text{ even}}} \frac{2}{j},
\end{align}
where \eqref{ln:rearrange} follows from the previous step by observing that each $\frac{1}{j}$ is counted \begin{equation}\#\{l \text{ odd}:j<\ell<L\}\ =\ \frac{L-j-1}{2}\end{equation} many times. This completes the induction step, so we have shown \eqref{eq:expectation}.
%\xinyu{I laid out the full calculation just so that the reader sees what is done exactly in each step. Might need to skip some of the steps in the final version to make it more concise?}
Now since 
\begin{equation}
\sum_{\substack{0< j < L \\ j \text{ even}}}\frac{1}{j}\ \sim \ \frac{1}{2}\log(L/2),
\end{equation}
we have
\begin{equation}
    \E[F_L]\ \sim\ \log L+O(1).
\end{equation}
By Markov's inequality,
\begin{equation}
\prob(F_L>g(L)\log L)
\ \leq\ \frac{\log L+O(1)}{g(L)\log L}
\ =\ O\left(\frac{1}{g(L)}\right).
\end{equation}
The proof of the lower bound follows the exact same reasoning as the proof of Lemma 3.4 in \cite{B+}.
% We now prove the lower bound. First we claim that for any $K$ with $\ell_K\geq L^{1/2}$,  with probability going to $1$, we have that $X_{k+1}\leq\frac{\ell_k}{\log L}$  for all $K\leq k< K+\log(\log L)^2$.
% This is because at each break $k$, the probability of breaking off $X_{k+1}>\frac{\ell_k}{\log L}$ is $\frac{}{}$
% Assuming there are at least $\log(\log L)^2$ breaks and that none of them breaks off a fragment larger than $\frac{\ell_k}{\log L}$ (where $\ell_k$ is the length of the active stick at that level), then since
% \begin{equation}
% \ell\cdot\left(\frac{1}{\log(\ell)}\right)^{\log(\log(\ell))^2}\geq \ell^{1/2}.
% \end{equation}
\end{proof}

\begin{corollary}
    \label{cor:almostallkaregood}
     Let $k_L$ be the total number of sticks when the process $\mathcal{P}$ ends. Let 
    \begin{equation}
    k_L'\ =\ |\{k: \ell_k\geq\log^3(L)\}|.
    \end{equation}
    Then with probability going to 1, 
    \begin{equation}
    \lim_{L\to\infty} \frac{k_L'}{k_L} \ = \ 1.
    \end{equation}
    Moreover, for all $k$ such that $\ell_k\geq\log^3(L)$, we have $Y_{k+1} \to \infty$ as $L \to \infty$ uniformly with probability going to $1$.
\end{corollary}
\begin{proof}
    The following argument is essentially the same as the one given in the proof of \cite[Corollary 3.5]{B+}. We include it here for completeness.
    Note that $k_L - k_L'$ is the number of sticks generated after $\ell_k$ first becomes smaller than $\log^3(L)$, and is thus upper bounded by $\log(\log^3(L))g(\log^3(L))$ with probability going to 1 by \Cref{lem:countfragments}. On the other hand, $k_L>(\log\log L)^2$ with probability going to 1. Therefore as $g(L)=o(\sqrt{\log (L)})$,
    \begin{equation}
    \lim_{L\to\infty}\frac{k_L'}{k_L}\ =\ 1-\lim_{L\to\infty}\frac{k_L-k_L'}{k_L}\ >\ 1- \frac{\log(\log^3(L))g(\log^3(L))}{(\log\log (L))^2}\ =\ 1
    \end{equation}
    with probability going to 1. 
    To prove the second part of the Corollary, note that, for $k$ such that $\ell_k \ge \log^3(L)$,
    \begin{equation}
    c_{k} \ge \frac{1}{g(L)\log^2(L)} \quad \implies \quad X_{k+1} \ge \frac{\log L}{g(L)}
    \end{equation}
    which approaches infinity. The probability of the former occurring for all such $k$ is
    \begin{equation}
    \left( 1 - \frac{1}{g(L)\log^2(L)} \right)^{g(L)\log L} \ = \ (1 - o(1))e^{-\frac{1}{\log L}} \ \to \ 1.
    \end{equation}
    Thus we immediately deduce the same holds for $Y_{k+1}$ in view of \Cref{lem:XandYclose}. This completes the proof.
\end{proof}

We have verified that all conditions required in \Cref{cor:XandYclose} are satisfied with probability going to 1, so we are done.

%%%%%%%%%%%%%%%%%%%%%%%%%%%%%%%%%%%%%%%%%%%%%%%%%%%%%%%%%%%
%%%%%%%%%%%%%%%%%%%%%%%%%%%%%%%%%%%%%%%%%%%%%%%%%%%%%%%%%%%
\subsection{Proof of Theorem \ref{thm:discreteStopAtHalfResidues}}
For any integer $\ell>1$, $r\in\{0,\dots,n-1\}$, let 
\begin{equation}
p_r(\ell) \ = \ \frac{|(n\Z+r)\cap [1,\dots,\ell-1]|}{\ell-1}.
\end{equation}
In other words, $p_r(\ell)$ is the proportion of integers between $1$ and $\ell-1$ falling into the residue class $r$ modulo $n$. Note that
\begin{equation}
\frac{1}{n} - \frac{1}{\ell-1} \ \le \ p_r(\ell) \ \le \ 
\frac{1}{n} + \frac{1}{\ell-1}
\end{equation}
for all $r$.
Define a discrete distribution $\mathcal{D}_\ell$ on $\{0,\dots,n-1\}$ by
\begin{equation}
\prob(X_\ell = r) \ = \ p_r(\ell). 
\end{equation}

Fix starting stick length $L\in\Z_+\backslash\mathfrak{S}$. We define a discrete process $\mathcal{P}$, and a continuous process $\mathcal{Q}$ that depends on $\mathcal{P}$ as follows. 
\begin{itemize}
    \item In both processes, we start with a stick of the same integer length $L>1$. Both starting sticks are assumed to be alive. (Since we are defining the process recursively, assume that at the start of each level, every living stick in $\mathcal{Q}$ uniquely corresponds to a living stick in $\mathcal{P}$ and vice versa. This is clearly true in the first level. We will see from our construction that this property is always preserved.)
    \item At each level, for each living stick in $\mathcal{P}$ of length $\ell$, choose a random ratio $p \in (0, 1)$ uniformly and a residue class $r\in\{0,\dots,n-1\}$ with respect to the distribution $\mathcal{D}_\ell$. Suppose $m = |(n\Z+r)\cap [1,\dots,\ell-1]|$. Let $X$ be the $(\lfloor mp \rfloor +1)$-th smallest integer in $[1,\dots,\ell-1]$ with residue $r$ modulo $n$. 
    \item Cut the stick in $\mathcal{P}$ into pieces of lengths $X$ and $\ell-X$, and cut the corresponding stick (of length $h$) in $\mathcal{Q}$ into pieces of lengths $ph$ and $(1-p)h$.
    \item Now, in process $\mathcal{P}$,  any new stick generated becomes dead if its length is in $\mathfrak{S}$, and in this case the corresponding stick in $\mathcal{Q}$ dies, too.  
    \item Continue to the next level until all sticks die. 
\end{itemize}
By choosing the ratio $p$ and the residue class $r$ of $X$ independently, we ensure that dead/alive status of a new stick in either process is independent of the ratio $p$ used to generate its length. In particular, in the continuous process, the probability that a new stick dies is always close to $1/2$ with an error of at most $\frac{n+4}{2(\ell-1)}$ (sum over $n/2$ residues and then an error of $\frac{2}{\ell-1}$ to account for stopping at length $1$).

We want to argue the following:
\begin{enumerate}
    \item The continuous process $\mathcal{Q}$ thus constructed is ``close'' to the process in \Cref{conj: contbreaks_dependent}, and thus results in strong Benford behavior.
    \item For almost all pairs of corresponding ending sticks $X_k$, $Y_k$ in $\mathcal{P}$, $\mathcal{Q}$ respectively, we have
    \begin{equation}
    Y_k(1-o(1)) \ \leq \ X_k \ \leq \ Y_k(1+o(1))
    \end{equation}
    as $L\to\infty$,
    so that we can apply \Cref{lem:smallPerturbPresBenford} to argue that $\mathcal{P}$ is Benford.
\end{enumerate}

\subsubsection{Proof of First Item}

\begin{lemma} \label{lem:live_stick_lower_bound}
    Let $T_i$ be the number of living sticks at level $i$ of length at least $\frac{L}{(\log L)^i}$. Given that $L > (n+5)(\log L)^{j}$, $\log L > 10j$, and $R > 2(\log L)^2$, we have that
    \begin{equation}
    \prob\left(T_i \ge R\left( 1 - \frac{5i}{\log L} \right) \ \forall \ 0 \le i \le j\right) \ \ge \ \left(1 - \frac{2(\log L)^2}{R} \right)^j \ \ge \ 1 - \frac{2j(\log L)^2}{R}.
    \end{equation}
\end{lemma}
\begin{proof}
    We proceed with induction on $j$. The result is clearly true for $j =0$ since $T_0 = R$. We show the result for $j$ implies that for $j+1$. From now on we condition on the history up to the $j$-th level. Consider a stick at level $j$ of length at least $\frac{L}{(\log L)^j}$. For each of its children, the probability of being shorter than $\frac{L}{(\log L)^{j+1}}$ is at most $\frac{1}{\log L}$ and the probability of being alive is at least
    \begin{equation}
    \frac{1}{2} - \frac{n+4}{2\left( \frac{L}{(\log L)^j} - 1 \right)} \ \ge \ \frac{1}{2} - \frac{n+4}{2(n+4)\log L} \ = \ \frac{1}{2} - \frac{1}{2\log L}.
    \end{equation}
    So the probability that the child is both of length at least $\frac{L}{(\log L)^{j+1}}$ and active is bounded below by
    \begin{equation}
    \prob(\text{alive}) - \prob\left(\text{length} \le \frac{L}{(\log L)^{j+1}}\right) \ge \frac{1}{2} - \frac{1}{2\log L} - \frac{1}{\log L} = \frac{1}{2} - \frac{3}{2\log L}.
    \end{equation}
    Let $T_{j+1}'$ be the number of live sticks at level $j+1$ of length at least $\frac{L}{(\log L)^{j+1}}$ whose parent is of length at least $\frac{L}{(\log L)^{j}}$. Since there are $T_j$ such parents generating $2T_j$ children in total, summing the above probability over each child, we have that
    \begin{equation}
    \E(T_{j+1}'|T_j) \ \ge \ 2T_j \left( \frac{1}{2} - \frac{3}{2\log L} \right) \ \ge \ T_j \left( 1 - \frac{3}{\log L}\right).
    \end{equation}
    Moreover, for each parent, the variance of the number of its active children that are of length at least $\frac{L}{(\log L)^{j+1}}$ is at most $2^2 = 4$ since it has at most 2 children in total. Also, each sub-process starting from one of these parents is independent from another, so the total variance $\Var(T_{j+1}') \le 4T_j$. Then, conditioning on the history of the process up to the $j$-th level, by Chebyshev's inequality,
    \begin{equation}
    \prob\left(T_{j+1} < T_j \left( 1 - \frac{5}{\log L} \right) \right) \ \le \ \prob\left( |T_{j+1}'-\E(T_{j+1}')| >  T_j \frac{2}{\log L} \right) \ < \ \frac{4T_j}{\frac{4T_j^2}{(\log L)^2}} \ < \ \frac{2(\log L)^2}{R},
    \end{equation}
    where the last inequality is true with probability $\left(1 - \frac{2(\log L)^2}{R} \right)^j$ by the induction hypothesis. This implies
    \begin{equation}\prob\left(T_{j+1} \ \ge\ T_j \left( 1 - \frac{5}{\log L} \right) \right)
    \ \ge\ 1-\frac{2(\log L)^2}{R}.\end{equation}
    Notice that
    \begin{equation}
    \left(1-\frac{5j}{\log L}\right)\left(1-\frac{5}{\log L}\right) \ >\ 1-\frac{5(j+1)}{\log L},
    \end{equation}
    so that 
    \begin{equation}
    T_{j+1}\ \ge\  T_j \left( 1 - \frac{5}{\log L} \right)
    \implies T_{j+1} \ \ge\  R\left( 1 - \frac{5(j+1)}{\log L} \right)
    \end{equation}
    given that $T_j\geq R(1-\frac{5j}{\log L})$.
    Thus we have 
    \begin{equation}
    \prob\left(T_{j+1} \ \ge\  R\left( 1 - \frac{5(j+1)}{\log L} \right) \right) \
    \ge\ 1-\frac{2(\log L)^2}{R}
    \end{equation}
    given that $T_j\geq R(1-\frac{5j}{\log L})$.
    This completes the induction step.
\end{proof}

\begin{corollary}
\label{cor:numStickLower}
    For sufficiently large $L$ and $R > (\log L)^3$, the number of live sticks ever is bounded below by
    \begin{equation}
    R \sqrt{\log L}
    \end{equation}
    with probability at least
    \begin{equation}
    1 - \frac{4(\log L)^{5/2}}{R}.
    \end{equation}
    This is also a lower bound for the number of dead sticks ever, with the same probability.
\end{corollary}
\begin{proof}
    Let $n_i$ be the number of active sticks at level $i$.
    First, note that the number of dead sticks generated at level $i$ is $2n_{i-1} - n_i$, and summing this from $i = 1$ to infinity yields $2n_0 + n_1 + n_2 + \cdots$ which is bounded below by the total number of live sticks in all levels. Now our lower bound follows from \Cref{lem:live_stick_lower_bound} by taking $j = \lfloor 2 \sqrt{\log L} \rfloor$ 
    and $L$ large enough so that all assumptions there hold and that 
    \begin{equation}
    \frac{5\cdot\lfloor 2 \sqrt{\log L} \rfloor}{\log L} \ < \ \frac{1}{2}.
    \end{equation}
\end{proof}

Let $M$ be the total number of dead sticks.
We have that, with probability going to 1, $M \ge R\sqrt{\log L}$. Now, we wish to show that $\E[P_{R,L}(s)] \to \log_B(s)$. It suffices to show that $\E[\varphi_s(X_i)] \to \log_B(s)$ uniformly for a proportion of $X_i$ going to $1$. We first show that almost all sticks die after $\frac{1}{2}\log\log L$ levels. First, note that there are at most $R2^i$ alive sticks at level $i$ and thus at most $R2^{i-1}\cdot 2 = R2^i$ new dead sticks are generated at level $i$. Thus, the number of dead sticks generated at or before level $j$ is $\sum_{i=1}^j R2^i \le R2^{j+1}$. Thus, the number of sticks before level $\frac{1}{2}\log \log L$ is at most
\begin{equation}
2R \cdot 2^{(\log \log L)/2} \ = \ 2R(\log L)^{(\log 2)/2} \ = \ o(R\sqrt{\log L}).
\end{equation}
That is, the proportion of sticks before level $\frac{1}{2}\log \log L$ goes to $0$. Thus, we may assume that $X_i$ dies at a later level. However, it is a product of independent random variables, each chosen from some finite set. Moreover, the length of this product is increasing in $L$, so by \Cref{thm:Mellin_transform_condition_continuity}, $\E[\varphi_s(X_i)]$ approaches $\log_B(s)$ uniformly, and the conclusion follows.

We now wish to show that $\Var[P_{R,L}(s)] \to 0$. The same strategy as in the proof of \Cref{conj:contbreakwstopratio} works with slight modifications that we highlight below.  Recall that the goal is to show that
\begin{equation}
\frac{1}{M^2}\sum_{i, j} \E[\varphi_s(X_iX_j)] \ \to \ \log_B(s)^2,
\end{equation}
where $X_i,X_j$ denote a pair of dead sticks.
Based on our observation above, we may restrict our attention to the collection of pairs of sticks only involving those that die after at least $\frac{1}{2}\log\log L$ levels. Note that running the exact same argument as in the proof of \Cref{conj:contbreakwstopratio} with $k=2$, we obtain that the number of pairs with high dependency (as describe in (2) in that proof) is bounded above by
$M(\log R)^{1 + \log 2}=o(M^2)$, so we are done.

\subsubsection{Proof of Second Item}

\begin{lemma}
\label{lem:ratiosErrorBound}
    At each level of $\mathcal{P}$, given that a stick of length $\ell$ breaks into sticks of lengths $X$ and $\ell-X$, with ratio $p$ in process $\mathcal{Q}$, we have
    \begin{equation}
        \left|\frac{X}{\ell}-p\right| \ \leq \ \frac{n+1}{\ell}.
    \end{equation}
    This also implies that 
    \begin{equation}
        \left|\frac{\ell-X}{\ell}-(1-p)\right| \ \leq \ \frac{n+1}{\ell},
    \end{equation}
    so we have the same bound for the error between the corresponding ratios in $\mathcal{P}$ and $\mathcal{Q}$ regardless of which child we look at.
    \begin{proof}
        We prove this for $r\neq 0$. Since 
        $
        m=\left\lfloor \frac{\ell-1-r}{n}\right\rfloor + 1$,
        we have
        \begin{align*}
            \frac{\ell-1-r}{n} \ \leq \ m \ \leq \ \frac{\ell-1-r}{n}+1.
        \end{align*}
        % \begin{align*}
        %     \frac{\ell-1-r}{n}\leq & m \leq \frac{\ell-1-r}{n}+1\\
        %     \implies
        %     p\frac{\ell-1-r}{n}\leq & pm \leq p(\frac{\ell-1-r}{n}+1)\\
        %     \implies p(\ell-1-r)-n\leq & \lfloor pm\rfloor n \leq p(\ell-1-r)+pn\\
        % \end{align*}
        Now $X = \left\lfloor pm \right\rfloor n + r$ whenever $r\neq 0$. (Note that here if $r=0$, we have $m=\lfloor\frac{\ell-1}{n}\rfloor$ and $X = \left\lfloor pm \right\rfloor n + n$ instead.) So
        \begin{equation*}
        p(\ell-1-r)-n + r \ \leq \ X \ \leq \ p(\ell - 1 - r) + pn + r
        \end{equation*}
        \begin{equation}
        \implies 
        p - \frac{p+pr+n-r}{\ell} \ \leq \ \frac{X}{\ell} \ \leq \ p + \frac{pn + r - p - pr}{\ell}.
        \end{equation}
        Notice that 
        \begin{equation}
        |p+pr+n-r| \ = \ |n+p-(1-p)r| \ \leq \ n+1 
        \end{equation}
        and 
        \begin{equation}|pn + r - p - pr| \ = \ |p(n-1)+(1-p)r| \ \leq \ n,\end{equation}
        so we have the desired. 
        One easily verifies the result for $r=0$ following a similar calculation.
        % We include the calculations in the $r=0$ case here for reference. Since $m=\lfloor \frac{\ell-1}{n}\rfloor$, we have
        % \begin{equation}
        % \frac{\ell-1}{n}-1\leq m\leq\frac{\ell-1}{n}
        % \end{equation}
        % \begin{equation}
        % \implies p\ell-(p+pn)\leq X\leq p\ell+(n-p)
        % \end{equation}
        % \begin{equation}
        % \implies
        % p-\frac{p+np}{\ell}\leq\frac{X}{\ell}\leq p+\frac{n-p}{\ell}
        % \end{equation}
        % \begin{equation}
        % \implies p-\frac{n+1}{\ell}\leq\frac{X}{\ell}\leq p+\frac{n}{\ell}.
        % \end{equation}        
    \end{proof}
\end{lemma}
\begin{corollary}
Consider a pair of sticks $(\ell_j,h_j)$ at level $j\geq 1$, where $\ell_j$ is in process $\mathcal{P}$ and $h_j$ is the corresponding one in process $\mathcal{Q}$. Denote their ancestors as $(\ell_i,h_i)$ for $0\leq i\leq j-1$, with $\ell_0=h_0 = L$. Suppose $h_{i+1} = p_ih_i$ for all $0\leq i\leq j-1$. Then we have
    \begin{equation}
    h_j\prod_{i=0}^{j-1}\left(1-\frac{n+1}{p_i\ell_i}\right) \ \leq \
    \ell_j \ 
    \leq \ h_j\prod_{i=0}^{j-1}\left(1+\frac{n+1}{p_i\ell_i}\right).
    \end{equation}
    \label{cor:h_jAndl_jRelation}
    \begin{proof}
        By \Cref{lem:ratiosErrorBound}, we have for all $1\leq i\leq j$,
        \begin{equation}
            p_{i-1}\left(1-\frac{n+1}{p_{i-1}\ell_{i-1}}\right) \ \leq \ \frac{\ell_i}{\ell_{i-1}} \ \leq \ p_{i-1}\left(1+\frac{n+1}{p_{i-1}\ell_{i-1}}\right),
            \label{eq:ratioBetweenAdjacentElls}
        \end{equation}
        and the corollary follows by taking the product over all such $i$.
    \end{proof}
\end{corollary}

\begin{corollary}
\label{cor:relationEll_jH_j}
    \begin{equation}
    h_j\prod_{i=1}^{j}\left(1-\frac{n+1}{\ell_i -n-1}\right) \ \leq \
    \ell_j
    \ \leq \ h_j\prod_{i=1}^{j}\left(1+\frac{n+1}{\ell_i -n-1}\right).
    \end{equation}
    \begin{proof}
        This follows from \Cref{cor:h_jAndl_jRelation} using the lower bound 
        \begin{equation}
        p_{i-1}\ell_{i-1} \ \geq \ \ell_i-n-1
        \end{equation}
        which follows from \Cref{lem:ratiosErrorBound}.
    \end{proof}
\end{corollary}

\begin{lemma}
\label{lem:closenessOfPAndQ}
    Let $f(L),g(L),h(L)$ be some functions in $L$ that go to infinity as $L\to\infty$ with $g(L)=o(f(L))$. Then for any dead stick $\ell_j$ with $j<g(L)$, if $\ell_j>f(L)+n+1$ and the corresponding sticks $h_j>h(L)$,  we have
    \begin{equation}
    h_j(1-o(1)) \ \leq \ \ell_j \ \leq \ h_j(1+o(1)).
    \end{equation}
    \begin{proof}
    From \Cref{cor:XandYclose}, we have
    \begin{equation}
    \ell_j \ \ge \ h_j\left( 1 - \sum_{i=1}^j \frac{n+1}{\ell_i-n-1} \right) \ \ge \ h_j\left( 1 - g(L)\frac{n+1}{f(L)-n-1}\right) \ = \ h_j(1 - o(1)).
    \end{equation}
    For the upper bound, we have that
    \begin{equation}
    \ell_j \ \le \ h_j\prod_{i=1}^{j}\left(1+\frac{n+1}{\ell_i -n-1}\right) \ \le \ h_j\left( 1 + \frac{n+1}{f(L)-n-1} \right)^{g(L)}.
    \end{equation}
    As $L \to \infty$, the expression above multiplying $h_j$ approaches
    \begin{equation}
    \lim_{L \to \infty}\exp\left( g(L)\frac{n+1}{f(L)-n-1} \right)\ =\ 1
    \end{equation}
    so $\ell_j \le h_j(1 + o(1))$, as desired.
    \end{proof}
\end{lemma}

Now the goal is to determine $f$ and $g$ so that 
\begin{equation}
\prob(\text{A stick dies within $g(L)$ levels}) \ = \ 1-o(1)
\end{equation}
and 
\begin{equation}
\lim_{L\to\infty}\frac{\#\text{dead sticks with length larger than }f(L)+n+1}{\#\text{all dead sticks}} \ = \ 1.
\end{equation}
The intuition is that we want to show most sticks die within the first $g(L)$ levels, and that most sticks that ever occur are long, i.e., larger than $f(L)$.

\begin{lemma}
\label{lem:upperBoundNumSticks}
Let $M$ be the number of dead sticks ever in a process starting with $R$ sticks of length $L$. Then
    \begin{equation}
    \prob(M<R(\log L)\nu(L)) \ \to \ 1
    \end{equation}
    as $L\to\infty$, where $\nu(L)$ is any function that goes to infinity as $L\to\infty$.
\begin{proof}
    Let $M_L$ be the number of dead sticks resulting from the process of breaking a single stick of length $L$. We have that $M_L=1$ whenever $L\in \mathfrak{S}$. 
    We prove by induction on $L$ that when $L\notin\mathfrak{S}$
    \begin{equation}
    \label{eq:expNumSticks}
        \E[M_L] \ \leq \ 6n^2\log L.
    \end{equation}
    (Here $\log(x)$ is short-hand for $\log_e(x)$.)
    When $1< L\leq 3n^2$ this is clear since 
    \begin{equation}
    6n^2\log L \ \geq \ 3n^2 \cdot 2\log(2) \ \geq \ 3n^2 \ \geq \ L
    \end{equation}
    and $M_L\leq L$ for all $L$.
    
    When $L>3n^2$ and $L\notin\mathfrak{S}$, we have 
    \begin{align}
        \E[M_{L}] \ &= \ \frac{1}{L-1} \nonumber
        \sum_{1\leq \ell\leq L-1}(\E[M_{\ell}]+\E[M_{L-\ell}])\\ \nonumber
        &= \ \frac{2}{L-1}\sum_{1\leq \ell\leq L-1}\E[M_{\ell}]\\ \nonumber
        &\leq \ \left(\frac{2}{L-1}\sum_{\underset{\ell\in\mathfrak{S}}{1\leq \ell\leq L-1}} 1 \right) +
        \left(\frac{2}{L-1}\sum_{\underset{\ell\notin\mathfrak{S}}{1\leq \ell\leq L-1}}6n^2\log(\ell)\right)\\ \nonumber
        &\leq \ \frac{2}{L-1}\left(\frac{L-1}{2}+\frac{n}{2}+1\right)+ \frac{12n^2}{L-1}\log\left(\prod_{\underset{\ell\notin\mathfrak{S}}{1\leq \ell\leq L-1}}\ell\right) \\ \nonumber
        &\leq \ 1+\frac{n+2}{L-1}+6n^2\log\left(\left(\prod_{\underset{\ell\notin\mathfrak{S}}{1\leq \ell\leq L-1}}\ell\right)^{\frac{2}{L-1}}\right) \\
        &\leq \ 1+ 6n^2\log L.
    \end{align}
    Note that we used the fact that 
    \begin{equation}
    |[1,L-1]\cap\mathfrak{S}| \ \leq \ \left(\left\lfloor\frac{L-1}{n}\right\rfloor+1\right)\cdot\frac{n}{2}+1 \ \leq \ \frac{L-1}{2}+\frac{n}{2}+1.
    \end{equation}
    To see the last inequality, 
    \begin{align}
    1+\frac{n+2}{L-1}+6n^2&\log\left(\left(\prod_{\underset{\ell\notin\mathfrak{S}}{1\leq \ell\leq L-1}}\ell\right)^{\frac{2}{L-1}}\right) \nonumber
        \ \leq \ 1+ 6n^2\log L\\ \nonumber
    &\iff \frac{n+2}{6n^2} + \log\left(\left(\prod_{\underset{\ell\notin\mathfrak{S}}{1\leq \ell\leq L-1}}\ell\right)^{2}\right) \ \le \ (L-1)\log L \\ \nonumber
    &\iff e^{(n+2)/(6n^2)} \ \leq \ \frac{L^{L-1}}{\left(\prod_{\underset{\ell\notin\mathfrak{S}}{1\leq \ell\leq L-1}}\ell\right)^{2}}\\ 
    & \ \Longleftarrow \ e^{1/(3n)}n^n \ \leq \ \frac{L^{L-1}}{\left(\prod_{\underset{\ell\notin\mathfrak{S}}{n+1\leq \ell\leq L-1}}\ell\right)^{2}}.
    \end{align}
    Note that on the RHS, the product has at most $\frac{L-1}{2}$ terms and the first $n/2$ terms are at most $2n$, so we have
    \begin{equation}
    \frac{L^{L-1}}{\left(\prod_{\underset{\ell\notin\mathfrak{S}}{n< \ell\leq L-1}}\ell\right)^{2}} \ \geq \ \frac{L^n}{(2n)^n }\cdot\frac{L^{L-n-1}}{\left(\prod_{\underset{\ell\notin\mathfrak{S}}{2n<\ell\leq L-1}}\ell\right)^{2}} \ \geq \ (3n)^n \ \geq \ e^{1/(3n)}n^n.
    \end{equation}
    Thus the induction step is complete.
    By Markov's inequality,
    \begin{equation}
    \prob(M>R(\log L)\nu(L)) \ \leq \ \frac{R\E[M_L]}{R(\log L)\nu(L)} \ \leq \ \frac{6n^2R\log L}{R(\log L)\nu(L)} \ = \ O\left(\frac{1}{\nu(L)}\right) \ \to \ 0
    \end{equation}
    as $L\to\infty$.
\end{proof}
\end{lemma}
Since at each level before the process ends, the number of sticks increase by at least 1, we have that the total number of levels at most $R(\log L)\nu(L)$ with probability going to $1$ as $L\to\infty$. Thus we can take 
\begin{equation}
g(L) \ = \ R(\log L)\nu(L).
\end{equation}

\begin{lemma}
\label{lem:shortDeadSticks}
Let $M_{\ell, k}$ denote the number of dead sticks with length smaller than $k$ coming from a process starting with a stick of length $\ell$. Let $c = 24 n^2$. Then for any $k \geq 2n ,\ell > 1$, we have
\begin{equation}
\E[M_{\ell, k}] \ \leq \ c\log(k).
\end{equation}
In particular,
\begin{equation}
\E[M_{L,\log^2(L)}] \ \leq \ 2c\log\log L.
\end{equation}
\begin{proof}
    When $\ell\leq k$, we have trivially
    \begin{equation}
    \E[M_{\ell, k}]\ =\ \E[M_{\ell}]\ \leq\ \frac{c}{4}\log(\ell)\ \leq\  \frac{c}{4}\log(k)
    \end{equation}
    by \eqref{eq:expNumSticks}. When $\ell>k$ and $\ell\in\mathfrak{S}$, we have 
    \begin{equation}
    \E[M_{\ell, k}]\ =\ 0.
    \end{equation}
    For a fixed $k$, we prove the result by induction on $\ell$.
    \begin{align}
        \E[M_{\ell,k}] \ &= \ \frac{1}{\ell-1}
        \sum_{1\leq x\leq \ell-1}(\E[M_{x,k}]+\E[M_{\ell-x, k}])\nonumber\\
        &= \ \frac{2}{\ell-1}\sum_{1\leq x\leq \ell-1}\E[M_{x,k}]\nonumber\\
        &= \ \frac{2}{\ell-1}\left(\sum_{1\leq x\leq k}\E[M_{x,k}]+\sum_{\underset{k< x\leq \ell-1}{x\notin\mathfrak{S}}}\E[M_{x,k}]\right)\nonumber\\
        &\le \ \frac{2}{\ell-1}\left(k\cdot\frac{c}{4}\log(k)+\frac{\ell-k-1+n}{2}\cdot c\log(k)\right)\nonumber\\
        &= \ c\log(k) \frac{\frac{1}{2}k+(\ell-k-1+n)}{\ell-1}\nonumber\\
        &\le\ c\log(k),
    \end{align}
    where the last step uses $\frac{k}{2}\geq n$.
\end{proof}
\end{lemma}

\begin{corollary}
\label{cor:mostSticksAreLong}
    Let $\ell_i$ denote a stick occurring at level $i$ in process $\mathcal{P}$. 
    \begin{equation}
    \frac{\#\{\ell_i>\log^2(L): i\leq g(L)\}}{\#\{\ell_i: i\leq g(L)\}} \ \to \ 1
    \end{equation}
    as $L\to\infty$ with probability going to 1.    
\end{corollary}
\begin{proof}
    We have from \Cref{lem:shortDeadSticks} that 
    \begin{equation}
    \E[\#\{\ell_i \le \log^2(L): i\leq g(L)\}]\ \le\ R\cdot 2c\log\log L.
    \end{equation}
    By Markov's inequality, 
    \begin{equation}
    \prob(\#\{\ell_i \le \log^2(L): i\leq g(L)\}>R(\log L)^{1/3})
    \ \leq\ \frac{R\cdot 2c\log\log L}{R(\log L)^{1/3}}
    \ \to\ 0
    \end{equation}
    as $L\to\infty$. In other words, 
    \begin{equation}\#\{\ell_i>\log^2(L): i\leq g(L)\} \ < \  R(\log L)^{1/3}\end{equation} with probability going to 1. On the other hand,
    by \Cref{cor:numStickLower}, we have as long as $R>(\log L)^3$,
    \begin{equation}
    \#\{\ell_i: i\leq g(L)\}\ \ge\  R(\log L)^{1/2}
    \end{equation}
    with probability going to 1.
    Since
    \begin{equation}
    \frac{R(\log L)^{1/3}}{R(\log L)^{1/2}}\ \to\  0
    \end{equation}
    as $L\to\infty$, we have the desired.
\end{proof}

Now to see that item (2) is true, it suffices to show that the premises of \Cref{lem:closenessOfPAndQ} are satisfied for most dead sticks. By \Cref{lem:upperBoundNumSticks}, almost all sticks die within the first $g(L) = R(\log L)\nu(L)$ levels, where $\nu(L)$ is any function that blows up as $L\to\infty$.  By \Cref{cor:mostSticksAreLong}, almost all dead sticks are at least $f(L)=\log^2(L)$ in length. We can choose $\nu$ such that $g(L)=o(f(L))$.
%Simulations show that the final stick lengths are distributed equally into bins of $[(\log L)^i, (\log L)^{i+1}]$

%Idea: Calculate the expected proportion of dead sticks in level $i$ that are shorter than $\log^2(L)$. The total expected proportion of dead sticks shorter than $\log^2(L)$ is the average of these proportions weighted by the number of dead sticks generated in each level. If we can show that the expected proportion is small, then we can use Markov's inequality to show that with probability 1 the actual proportion goes to 0. 

%%%%%%%%%%%%%%%%%%%%%%%%%%%%%%%%%%%%%%%%%%%%%%%%%%%%%%%%%%%
%%%%%%%%%%%%%%%%%%%%%%%%%%%%%%%%%%%%%%%%%%%%%%%%%%%%%%%%%%%
\subsection{When $|S| \neq n/2$}
\label{sec:discreteNonBenford}

It is natural to ask what happens when $|S|$ is not exactly $n/2$. In this section, we present a result showing that when $|S| < n/2$, the final stick lengths are non-Benford. Moreover, we state a more specific conjecture on the behavior of the limiting distribution when $|S| \neq n/2$. Simulation results are also presented to support our conjecture.

\begin{theorem}
\label{thm:|S|LessThanN/2}
    If $|S|<n/2$, then as $R\to\infty$ and $L\to\infty$, the collection of mantissas of ending stick lengths does not converge to any continuous distribution on $[0,1]$. In particular, it does not converge to strong Benford behavior.
\end{theorem}
% Quantify distribution based on the size of $|S|$, how far it is from $n/2$, $n\to\infty$? 
% $n$ growing? $n$ finite and fixed? $n$ independent of $R$ and $L$, growing with $R$/$L$?

\begin{theorem}
\label{thm:|S|GreaterThanN/2}
    If $|S| > n/2$, then as $R\to\infty$ and $L\to\infty$ with the condition $R = \omega(L^2)$, the collection of mantissas of ending stick lengths does not converge to the uniform distribution on $[0,1]$ provided that the base $B$ is greater than $3^{6n^3/|S|}$.
\end{theorem}

The above theorem says that in the case $|S|>n/2$, as long as the base $B$ is large enough, the final distribution does not converge to Benford. We conjecture that this is in fact true regardless of the base.

\begin{conjecture}[Other Bases]
    The final collection of stick lengths does not converge to strong Benford behavior \emph{for any base $B$} if the size of $S$ is not $n/2$. Specifically, if $|S|>n/2$, then the limiting distribution depends on the mantissa of $L$ base $B$, and the density function of $\log_B(X/L)\pmod 1$ is skewed towards 1.
\end{conjecture}

Note that the proofs of \Cref{thm:|S|GreaterThanN/2}
 already strongly indicate that the above conjecture is true, although it remains an interesting open question to describe precisely what the distribution looks like.  

 We have also obtained strong empirical evidence for the conjecture.
\Cref{fig:MoreThanHalfRes} shows a clear skewness of the limiting distribution after normalizing by the starting length $L$. In this simulation, we used inputs $n=12$, $\mathfrak{S}=\{1, 2, 3, 4, 5, 6, 9, 10\}$, $L=82\cdot 10^{12000}$ and $R=1000$. It is not hard to see that \Cref{thm:|S|GreaterThanN/2} does not apply when $B = 10$. It is worth noting that even when we vary the stopping set and the significand of $L$, the pattern persists, as long as $n/2<|\mathfrak{S}|<n$.

\begin{figure}[ht]
    \centering
    \includegraphics[width=15cm]{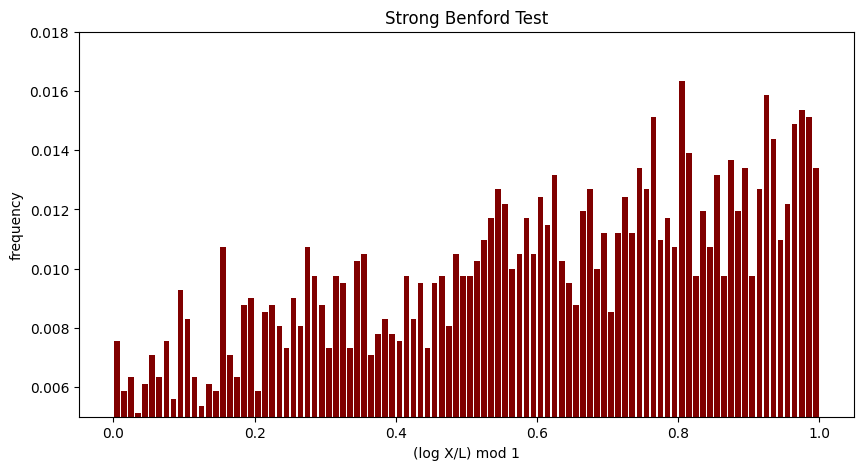}
    \caption{Stopping at 8 residues modulo 12 with $L=82\cdot 10^{12000}$ and $R=1000$}
    \label{fig:MoreThanHalfRes}
\end{figure}

A heuristic for this is that when $|\mathfrak{S}| = n$, all processes end at the first level, and the resulting stick lengths follow the uniform distribution on $\{1,\ldots,L-1\}$. It is not hard to show that for a random variable $X$ uniformly distributed, $\log(X)$ has smaller mantissa with lower probability and larger mantissa with higher probability.

\subsubsection{Proof of \Cref{thm:|S|LessThanN/2}}

\begin{lemma}
\label{lem:ExpShortSticks}
    Let 
    \begin{equation}
        M_{L,m} \ := \ \#\text{dead sticks generated by a stick of length $L$ that are of length less than $m$}.
    \end{equation}
    Then for all $L\notin\mathfrak{S}$, there exists constants $m,c$ only depending on $k$ and $n$ such that 
    \begin{equation}
        \E(M_{L,m})\ \ge\ c \E(M_L) + 1.
    \end{equation}
    \begin{proof}
        Let
        \[
        c = \frac{1}{2n+1}, \quad m = 2n^2.
        \]
        For $L \le m$, the result is clear since $\E(M_{L,m}) = \E(M_L) \ge 2$. We now proceed with induction and assume the result is true for positive integers less than $L > m$. We have,
        \begin{align}
            \E(M_{L,m}) \ &= \ \frac{2}{L-1}\sum_{\ell = 1}^{L-1} \E(M_{\ell,m}) \nonumber\\
            &= \ \frac{2}{L-1} \sum_{\substack{1 \le \ell \le L-1 \\ \ell \not\in \mathfrak{S}}} \E(M_{\ell, m}) \nonumber \\
            &\ge \ \frac{2}{L-1} \sum_{\substack{1 \le \ell \le L-1 \\ \ell \not\in \mathfrak{S}}} c\E(M_\ell) + \frac{2}{L-1}|[L-1] \setminus \mathfrak{S}|. \label{initial_expectation_calc_small_sticks}
            %&\ge \ \frac{2}{L-1} \sum_{\substack{1 \le \ell \le L-1 \\ \ell \not\in \mathfrak{S}}} c\E(M_\ell) + \frac{2}{L-1}(n-|S|)\left( \frac{L-1}{n}-1 \right)
        \end{align}
        We now show that
        \begin{equation} \label{stopping_set_size_comparison}
            |[L-1] \setminus \mathfrak{S}| \ \ge \ \frac{L-1}{2} + c|[L-1] \cap \mathfrak{S}|.
        \end{equation}
        Note that
        \begin{equation}
            |[L-1] \setminus \mathfrak{S}| \ \ge \ (n - |S|)\left( \frac{L-n-1}{n} \right) \ \ge \ \frac{(n+1)(L-n-1)}{2n}
        \end{equation}
        and
        \begin{equation}
            |[L-1] \cap \mathfrak{S}| \ \le \ |S|\left( \frac{L+n-1}{n} \right) \ \le \ \frac{(n-1)(L+n-1)}{2n}
        \end{equation}
        so it suffices to show that
        \begin{align}
            \frac{(n+1)(L-n-1)}{2n} \ \ge \ \ \frac{L-1}{2} + c\frac{(n-1)(L+n-1)}{2n} \nonumber\\
            \iff c \ \le \ \frac{(n+1)(L-n-1) - n(L-1)}{(n-1)(L+n-1)}.
        \end{align}
        We have
        \begin{align}
            \frac{(n+1)(L-n-1) - n(L-1)}{(n-1)(L+n-1)} \ & \ge \ \frac{L-1-n^2-n}{(n-1)(L-1+n)} \nonumber\\
            &\ge \ \frac{(2n^2-n^2-n}{(n-1)(2n^2+n-1)} \nonumber \\
            &\ge \ \frac{n}{2n^2+n-1} \nonumber\\
            &\ge \ \frac{1}{2n+1} \ = \ c.
        \end{align}
        Hence, \eqref{stopping_set_size_comparison} is true, and can be plugged into \eqref{initial_expectation_calc_small_sticks} to obtain
        \begin{align}
            \E(M_{L.m}) \ &\ge \ \frac{2}{L-1} \sum_{\substack{1 \le \ell \le L-1 \\ \ell \not\in \mathfrak{S}}} c\E(M_\ell) + \frac{2c}{L-1}|[L-1] \cap \mathfrak{S}| + 1 \nonumber \\
            &\ge \ 1 + \frac{2}{L-1} \sum_{\ell = 1}^{L-1} c\E(M_\ell) \nonumber \\
            &\ge \ c\E(M_L) + 1.
        \end{align}
        The induction is complete.
    \end{proof}
\end{lemma}

\begin{lemma}
Let $M_{L}^R$ be the total number of dead sticks coming from a process of breaking $R$ identical sticks of length $L$, and $M_{L,m}^R$ be the number of those shorter than $m$. Then for $m$ and $c$ satisfying the conclusion of \Cref{lem:ExpShortSticks}, we have as $R\to\infty$,
    \begin{equation}
        \prob\left(\frac{M_{L,m}^R}{M_{L}^R}\le \frac{c}{3}\right)\ \to\ 0.
    \end{equation}
    \begin{proof}
        By Chernoff's inequality, we have
        \begin{equation}
            \prob\left(M_{L,m}^R\le\frac{1}{2}R\E(M_{L,m})\right)\ \le\ e^{-R\E(M_{L,m})/8}
        \end{equation}
        and 
        \begin{equation}
            \prob\left(M_{L}^R\ \ge \frac{3}{2}R\E(M_{L})\right)\ \le\ e^{-R\E(M_{L})/10}.
        \end{equation}
        So
        \begin{equation}
            \prob\left(M_{L,m}^R\ \ge \frac{1}{2}R\E(M_{L,m}) \text{ and } M_{L}^R\ \le \frac{3}{2}R\E(M_{L})\right)\ \ge\ 1-e^{-R\E(M_{L,m})/8}-e^{-R\E(M_{L})/10}.
        \end{equation}
        In that case, 
        \begin{equation}
            \frac{M_{L,m}^R}{M_L^R} \ge \frac{\frac{1}{2}R\E(M_{L,m})}{\frac{3}{2}R\E(M_{L})} = \frac{\E(M_{L,m})}{3\E(M_{L})}\ \ge\ c/3.
        \end{equation}
        Therefore it suffices to show that 
        \begin{equation}
            1-e^{-R\E(M_{L,m})/8}-e^{-R\E(M_{L})/10}\ \to\ 0
        \end{equation}
        as $R\to\infty$. To do this, it again suffices to show that $\E(M_{L,m})>0$ and $\E(M_L)>0$. Clearly, $\E(M_L)\ge 1$, so it follows from \Cref{lem:ExpShortSticks} that $\E(M_{L,m})\ge c+1>0$. 
    \end{proof}
\end{lemma}

Now to conclude the proof of (2), note that if the limiting distribution were Benford, by \Cref{def:StrongBenfordForD}, the collection of mantissas $M_B(X)$ of dead sticks converges to uniform distribution on $[0,1]$, which is continuous. So we must have that 
\begin{equation}
    \frac{M_{L,m}}{M_L}\ =\ \frac{\# \{X:M_B (X)\in\{M_B(1),M_B(2),\dots,M_B(m-1)\}\}}{M_L}
    \ \to \ 0
\end{equation}
as $L\to\infty$ with probability going to 1 as $R\to\infty$. In fact, this shows that the collection of mantissas of such a process does not converge to \emph{any} continuous distribution on $[0,1]$ as $R\to\infty$ and $L\to\infty$.

\subsubsection{Proof of \Cref{thm:|S|GreaterThanN/2}}

Let $M_L^R$ be the total number of dead sticks obtained starting from $R$ sticks of length $L$.

\begin{lemma}
    We have that 
    \begin{equation}
        \E[M_L^R] \ \le \ 2n^2R.
    \end{equation}
\end{lemma}
\begin{proof}
    We show the result when $R = 1$ via induction on $L$. Let $M_L^1 = M_L$. The result is clearly true for $L \le 2n^2$ since $M_L \le L$, so assume that $L > 2n^2$ and the result holds for all positive integers smaller than $L$. We have that,
    \begin{align}
        \E[M_L] \ &= \ \frac{1}{L-1}\sum_{\ell = 1}^{L-1} (\E[M_{\ell}] + \E[M_{L-\ell}]) \nonumber\\ 
        &= \ \frac{2}{L-1} \sum_{\ell = 1}^{L-1} \E[M_{\ell}] \nonumber\\
        &\le \ 2 + \frac{2}{L-1} \sum_{\substack{1 \le \ell \le L-1 \\ \ell \not\in \mathfrak{S}}} 2n^2 \nonumber\\
        &\le \ 2 + \frac{2}{L-1} \left\lceil \frac{L-1}{n} \right\rceil (n-|S|) \cdot 2n^2\nonumber \\
        &\le \ 2 + \frac{2}{L-1}\left( \frac{L+n-1}{n} \right) \left( \frac{n-1}{2} \right) 2n^2 \nonumber\\
        &= \ 2 + 2n(n-1) \frac{L+n-1}{L-1}\nonumber \\
        &\le \ 2 + 2n(n-1) \frac{2n^2+n}{2n^2} \nonumber\\
        &= \ 2n^2 -n + 1 \ \le \ 2n^2.
    \end{align}
    The induction is complete. The result for general $R$ follows from linearity of expectation.
\end{proof}

\begin{corollary} \label{cor: S_large_sticks_small}
    With probability at least
    \begin{equation}
        1 - \frac{L^2}{n^4R}
    \end{equation}
    we have $M_L^R \le 3n^2R$.
\end{corollary}
\begin{proof}
    First, note that, trivially, $\Var[M_L] \le L^2$ so that $\Var[M_L^R] \le RL^2$. Then, Chebyshev's inequality implies that
    \begin{equation}
        \prob(M_L^R > 3n^2R) \ \le \ \prob(|M_L^R - \E(M_L^R)| > n^2R) \ \le \ \frac{RL^2}{(n^2R)^2} \ = \ \frac{L^2}{n^4R}.
    \end{equation}
\end{proof}

\begin{lemma} \label{lem: S_large_some_sticks_large}
    Let $a > 2$ be some real number and assume $L > \frac{2an}{a-2}$. With probability at least
    \begin{equation}
        1 - \frac{16n^2}{|S|^2R}
    \end{equation}
    the number of dead sticks in the first level of length at least $L/a$ is at least $|S|R/(2n)$.
\end{lemma}

\begin{proof}
    Denote this quantity by $M_{L,a}^R(1)$. Then, given a stick of length $L$, the number of ways the left child can die and be of length at least $L/a$ is bounded below by
    \begin{equation}
        |S|\left\lfloor \frac{L-L/a}{n} \right\rfloor \ \ge \ |S| \left( \frac{L}{n}\left(1-\frac{1}{a}\right) - 1 \right) \ \ge \ \frac{|S|L}{2n}.
    \end{equation}
    Thus, the probability of any arbitrary child being of length at least $L/a$ and dead is at least $|S|/(2n)$. It follows that
    \begin{equation}
        \E[M_{L,a}^R(1)] \ \ge \ \frac{|S|R}{n}.
    \end{equation}
    Furthermore,
    \begin{equation}
        \Var[M_{L,a}^R(1)] \ \le \ 4R
    \end{equation}
    by independence. Thus, we have, by Chebyshev's inequality, 
    \begin{equation}
        \prob\left(M_{L,a}^R(1) \le \frac{|S|R}{2n}\right) \ \le \ \prob\left( \left|M_{L,a}^R(1) - \E[M_{L,a}^R(1)]\right| \ge \frac{|S|R}{2n}\right) \ \le \ \frac{4R}{\left( \frac{|S|R}{2n}\right)^2} \ = \ \frac{16n^2}{|S|^2R}.
    \end{equation}
\end{proof}
Now, set $a = 3$ and let $L > 6n$. Note that by \Cref{lem: S_large_some_sticks_large} and \Cref{cor: S_large_sticks_small}, with probability at least
\begin{equation}
    1 - \frac{L^2}{n^4R} - \frac{16n^2}{|S|^2R}
\end{equation}
the proportion of dead sticks of length at least $L/3$ is bounded below by
\begin{equation}
    \frac{|S|R}{2n}(3n^2R)^{-1} \ = \ \frac{|S|}{6n^3}.
\end{equation}
As $L,R \to \infty$ in a manner such that $R$ grows faster than $L^2$, this probability approaches $1$. Now let, 
\begin{equation}
    B \ > \ 3^{6n^3/|S|}.
\end{equation}
We obtain that
\begin{equation}
    \log_B(3)\ < \ \frac{|S|}{6n^3}
\end{equation}
but at least $\frac{|S|}{6n^3}$ of the dead sticks are in $[L/3, L]$ so that at least the same fraction of normalized mantissas of dead sticks are in $[1 - \log_B(3), 1]$. It follows that the distribution of mantissas of dead sticks cannot approach the uniform distribution as $R \to \infty$ for any $L > 6n$, nor can such be the case as $L \to \infty$.

%%%%%%%%%%%%%%%%%%%%%%%%%%%%%%%%%%%%%%%%%%%%%%%%%%%%%%%%%%%
%%%%%%%%%%%%%%%%%%%%%%%%%%%%%%%%%%%%%%%%%%%%%%%%%%%%%%%%%%%
\subsection{General Number of Parts}
\label{sec:discreteGeneralParts}
Given \Cref{conj: contbreaks_dependent}, it seems likely that a similar result would hold for the discrete analogue. Indeed, we make the following conjecture, which is supported by our simulation results (see, for example,  \Cref{fig:generalnumofparts}).
\begin{figure}[ht]
    \centering
    \includegraphics[width=15cm]{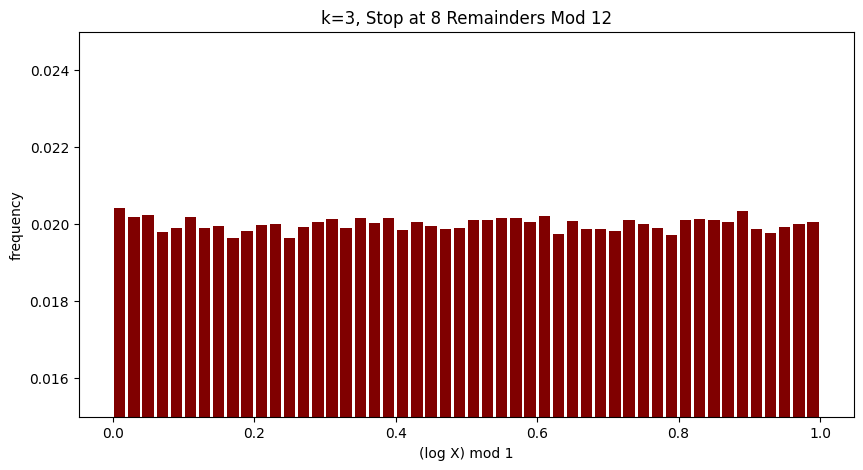}
    \caption{Break Into 3 Pieces and Stop at 2/3 of the Residue Classes.}
    \label{fig:generalnumofparts}
\end{figure}
\begin{conjecture}[General number of parts]
\label{conj:discreteGeneralK}
    Fix some positive integer $k\geq 2$, and consider the process where we break each stick into $k$ pieces by choosing $k-1$ cut points recursively following the uniform distribution\footnote{Namely, choose the first cut point according to the uniform distribution as usual, and then choose the next cut point on the second fragment according to the uniform distribution on that fragment, and so on. If at some point the second fragment has length 1, then the breaking stops - so when the stick is short, it is possible that it only breaks into less than $k$ pieces.}. Fix a modulus $n =  tk$ for some $t\geq 1$ and a subset $S\subset\{0,\dots,n-1\}$ of size $(t-1)k$ representing the residue classes. Let the stopping set be
    \begin{equation}
        \mathfrak{S} \ := \ \{1\}\cup\{m\in\Z_+ :m = qn+r,\  r\in S, q\in \Z\}.
    \end{equation}
    If we start with $R$ identical sticks of positive integer length $L\notin\mathfrak{S}$, then the collection of ending stick lengths converges to strong Benford behavior given that $R>f(L)$ as $L \to\infty$, where $f(L)$ is some function that goes to infinity as $L\to\infty$. Moreover, if the number of residue classes constituting the stopping set is not equal to $(t-1)k$, then the resulting stick lengths do not converge to strong Benford behavior.

\end{conjecture}

%%%%%%%%%%%%%%%%%%%%%%%%%%%%%%%%%%%%%%%%%%%%%%%%%%%%%%%%%%%%%%%%%%%%%%%%%%%%%%%%%%%%%%%%%%%%%%%%%%%%%%%%%%%%%%%%%%%%%%
%%%%%%%%%%%%%%%%%%%%%%%%%%%%%%%%%%%%%%%%%%%%%%%%%%%%%%%%%%%%%%%%%%%%%%%%%%%%%%%%%%%%%%%%%%%%%%%%%%%%%%%%%%%%%%%%%%%%%%

\section{Acknowledgements}
The authors are supported by NSF Grant DMS2241623, NSF Grant DMS1947438, Williams College, and University of Michigan. 

%%%%%%%%%%%%%%%%%%%%%%%%%%%%%%%%%%%%%%%%%%%%%%%%%%%%%%%%%%%%%%%%%%%%%%%%%%%%%%%%%%%%%%%%%%%%%%%%%%%%%%%%%%%%%%%%%%%%%%%%%%%%%%%%%%%%%%%%%%%%%%%%%%%%%%%%%%%%%%%%%%%%%%%%%%%%%%%%%%%%%%%%%%%%%%%%%%%%%%%%%%%%%%%%%%%%%%%%%%%%%%%%%%%%%%%%%%%%%%%%%%%%%%%%%%%%%%%%%%%%%%%%%%%%%%%%%%%%%%%%%%%%%%%%%%%%%%%%%%%%%%%%%%%%%%%%%%%%%%%%%%%%%%%%%%%%%%%%%%%%%%%%%%%%%%%%%%%%%%%%%%%%%%%%%%%%%%%%%%%%%%%%%%%%%%%%%%%%

%%%%%%%%%%%%%%%%%%%%%%%%%%%%%%%%%%%%%%%%%%%%%%%%%%%%%%%%%%%%%%%%%%%%%%%%%%%%%%%%%%%%%%%%%%%%%%%%%%%%%%%%%%%%
%%%%%%%%%%%%%%%%%%%%%%%%%%%%%%%%%%%%%%%%%%%%%%%%%%%%%%%%%%%%%%%%%%%%%%%%%%%%%%%%%%%%%%%%%%%%%%%%%%%%%%%%%%%%

\

\end{document}